\documentclass[11pt]{amsart}
\usepackage[latin1]{inputenc}
\usepackage[all]{xy}
\usepackage{graphics}
\usepackage{amscd}
\usepackage{amssymb}
\usepackage{amsthm}
\usepackage{enumitem}
\usepackage{comment}
\newtheorem{thm}{Theorem}[section]

\newtheorem{pro}[thm]{Proposition}
\newtheorem{lem}[thm]{Lemma}
\newtheorem{cor}[thm]{Corollary}
\newtheorem{defn}[thm]{Definition}

\newtheorem*{rem*}{Remarks}
\newtheorem{rems}[thm]{Remark}
\newtheorem*{conj*}{Conjecture}

\DeclareMathOperator{\A}{\mathbb{A}}
\DeclareMathOperator{\C}{\mathbb{C}}

\DeclareMathOperator{\Q}{\mathbb{Q}}
\DeclareMathOperator{\R}{\mathbb{R}}
\DeclareMathOperator{\Z}{\mathbb{Z}}

\DeclareMathOperator{\Hom}{Hom}

\DeclareMathOperator{\Res}{Res}

\DeclareMathOperator{\ad}{ad}

\DeclareMathOperator{\im}{Im}

\DeclareMathOperator{\ord}{ord}

\DeclareMathOperator{\re}{Re}

\newcommand{\lra}{\rightarrow}

\newcommand{\mrm}{\mathrm}
\newcommand{\mbb}{\mathbb}

\title{Algebraic cycles and residues of degree eight $L$-functions of $\mrm{GSp}(4) \times \mrm{GL}(2)$}
\author{Francesco Lemma}
\address{Univ Paris Diderot, Institut math\'ematique de Jussieu-Paris Rive Gauche, UMR 7586, B\^atiment Sophie Germain, Case 7012, 75205 Paris Cedex 13}
\email{francesco.lemma@imj-prg.fr}

\begin{document}

\begin{abstract}  We prove a cohomological formula for non-critical residues of degree eight automorphic $L$-functions of $\mrm{GSp}(4) \times \mrm{GL}(2)$ in the spirit of Beilinson conjecture. We rely on the cohomological interpretation of an automorphic period integral and on the study of Novodvorsky's integral representation of the $L$-functions. 
\,\\
\end{abstract}

\maketitle

\tableofcontents

\section{Introduction}

Let $X$ be a smooth projective variety over $\Q$ and let $n \geq 1$ be an integer. The Hasse-Weil $L$-function $L(s, H^{2n-2}(X))$ associated to the cohomology of $X$ in degree $2n-2$ is defined by an Euler product which is absolutely convergent for $\re(s)>n$ and which is expected to have a meromorphic continuation to the whole complex plane, the only possible pole occuring at $n$. Let $N^{n-1}(X)$ be the $\Q$-vector space of cycles of codimension $n-1$ on $X$ modulo homological equivalence.

\begin{conj*} (Tate) 
$$
-\ord_{s=n}L(s, H^{2n-2}(X))=\dim_{\Q} N^{n-1}(X).
$$
\end{conj*}

To give the geometric interpretation of the first non-zero term $L^*(n, H^{2n-2}(X))$ in the Taylor expansion of $L(s, H^{2n-2}(X))$ at $s=n$, let $H^{2n-1}_{\mathcal{M}}(X, \Q(n))_{\Z}$ and $H^{2n-1}_{\mathcal{D}}(X/\R, \R(n))$ denote the integral part of the motivic cohomology group $H^{2n-1}_{\mathcal{M}}(X, \Q(n))$ and the real Deligne-Beilinson cohomology respectively and let
$$
\begin{CD}
r_{\mathcal{D}}: H^{2n-1}_{\mathcal{M}}(X, \Q(n))_{\Z} \oplus N^{n-1}(X) @>>> H^{2n-1}_{\mathcal{D}}(X/\R, \R(n))
\end{CD}
$$
denote the thickened regulator. It is known that $H^{2n-1}_{\mathcal{D}}(X/\R, \R(n))$ is a finite dimensional $\R$-vector space and that there is a canonical isomorphism
$$
\mrm{Ext}^1_{\mrm{MHS}_{\R}^+}(\R(0), H_B^{2n-2}(X, \R(n))) \simeq H^{2n-1}_{\mathcal{D}}(X/\R, \R(n)).
$$
where $\mrm{MHS}_{\R}^+$ denotes the abelian category of  real mixed $\R$-Hodge structures, whose definition is recalled in the body of the article (Def. \ref{realmixed}). Let $\mathcal{D}(n, H^{2n-2}(X))$ denote the Deligne $\Q$-structure on the highest exterior power of $H^{2n-1}_{\mathcal{D}}(X/\R, \R(n))$.

\begin{conj*}(Beilinson)\\

i. The map $r_{\mathcal{D}}$ induces an isomorphism
$$
\begin{CD}
(H^{2n-1}_{\mathcal{M}}(X, \Q(n))_{\Z} \oplus N^{n-1}(X)) \otimes_{\Q} \R @>\sim>> H^{2n-1}_{\mathcal{D}}(X/\R, \R(n)),
\end{CD}
$$
\,\\
\indent ii. $\ord_{s=n-1}L(s, H^{2n-2}(X))=\dim_{\Q} H^{2n-1}_{\mathcal{M}}(X, \Q(n))_{\Z},$\\

iii. $\det (\im r_{\mathcal{D}})=L^*(n, H^{2n-2}(X))\mathcal{D}(n, H^{2n-2}(X)).$
\end{conj*}

For an introduction to this circle of ideas and for more details on the conjectures, the reader is referred to \cite{nekovar}, \cite{den-scholl} and \cite{schneider}.\\

To state our two main results, which are motivated by the conjectures above, let us denote by $\A$ the ring of adeles of $\Q$ and let us consider irreducible cuspidal automorphic representations $\Pi=\bigotimes'_v \Pi_v$ and $\sigma=\bigotimes'_v \sigma_v$ of $\mrm{GSp}(4, \A)$ and $\mrm{GL}(2, \A)$ respectively. Let $V$ be the finite set of places where $\Pi$ or $\sigma$ is ramified, together with the infinite place. We are interested in the $L$-function 
$$
L_V(s, \Pi \times \sigma)=\prod_{v \notin V} L(s, \Pi_v \times \sigma_v)
$$ 
which is associated to the tensor product eight-dimensional representation of the Langlands dual of $\mrm{GSp}(4) \times \mrm {GL}(2)$, which is $\mrm{GSp}(4, \C) \times \mrm {GL}(2, \C)$. This partial Euler product converges for $\re(s)$ big enough. To study special values of $L_V(s, \Pi \times \sigma)$ by motivic methods, we assume that the non-archimedean components $\Pi_f$ and $\sigma_f$ contribute to the cohomology of a Siegel threefold $S$ and of a modular curve $Y$ respectively, with trivial coefficient system. To use Novodvorsky's integral representation \cite{moriyama1}, \cite{soudry} we need to assume that $\Pi$ is globally generic. Examples of such automorphic representations are obtained as Weil liftings from $\mrm{GSO}(2, 2)$ (see \cite{harris-kudla}). Under the assumption that $\Pi$ is globally generic, the main result of \cite{moriyama1} applies, hence the partial Euler product $L_V(s, \Pi \times \sigma)$ can be completed to a product over all places which has a meromorphic continuation to the whole complex plane and a functional equation relating $s$ and $1-s$. Let $L(s, \Pi \times \sigma)$ denote the product completed to all non-archimedean places as defined in \cite{moriyama1}, \cite{soudry}. Because $\Pi_f$ and $\sigma_f$ are cohomological, they are defined over a number field. Let $E$ denote a fixed number over which both $\Pi_f$ and $\sigma_f$ are defined. For any $r \in \Z$, let $H^4_{B, !}(S \times Y, \Q(r))$ denote the image of compactly supported Betti cohomology in the cohomology without support, at infinite level, with coefficients in $\Q(r)$. Let
\begin{eqnarray*}
M_B(\Pi_f \times \sigma_f, r) &=& \mrm{Hom}_{E[G(\A_f)]}(\Pi_f \times \sigma_f, H^4_{B, !}(S \times Y, \Q(r)) \otimes_{\Q} E)
\end{eqnarray*}
where we denote by $G$ the reductive group $\mrm{GSp}(4) \times \mrm {GL}(2)$. This is a pure $\Q$-Hodge structure of weight $4-2r$, with coefficients in $E$. Let us denote by $(M_B(\Pi_f \times \sigma_f, 2)_{\R} \cap M^{0, 0})^+$ the vectors of $M_B(\Pi_f \times \sigma_f, 2)_{\R}$ which have Hodge type $(0, 0)$ and which are fixed by complex conjugation. We have a canonical isomorphism of $\R \otimes_{\Q} E$-modules
$$
\mrm{Ext}^1_{\mrm{MHS}_{\R}^+}(\R(0), M_B(\Pi_f \times \sigma_f, 3)_{\R}) \simeq (M_B(\Pi_f \times \sigma_f, 2)_{\R} \cap M^{0, 0})^+
$$ 
and we will deduce from the main result of \cite{jiang-soudry} that these $\R \otimes_{\Q} E$-modules have rank one. Let $\mathcal{D}(\Pi_f \times \sigma_f)$ denote the Deligne $E$-structure on  $
\mrm{Ext}^1_{\mrm{MHS}_{\R}^+}(\R(0), M_B(\Pi_f \times \sigma_f, 3)_{\R})$. Let us consider the group
$
\mathrm{GL}(2) \times_{\mathbb{G}_m} \mathrm{GL}(2)=\{(g, h) \in \mrm{GL}(2) \times \mrm{GL}(2) \,|\, \det(g)=\det(h) \}
$
and let 
$$
\iota: \mathrm{GL}(2) \times_{\mathbb{G}_m} \mathrm{GL}(2) \rightarrow \mrm{GSp}(4) \times \mrm{GL}(2)
$$
be the embedding defined by
$$
\iota\left( \begin{pmatrix}
a & b\\
c & d\\
\end{pmatrix}, \begin{pmatrix}
a' & b'\\
c' & d'\\
\end{pmatrix} \right) =
\left( \begin{pmatrix}
a & & b & \\
 & a' &  & b'\\
c &  & d & \\
 & c' &  & d'\\
\end{pmatrix},  \begin{pmatrix}
a' & b'\\
c' & d'\\
\end{pmatrix} \right).
$$
As explained in section \ref{section-cycle} below, the morphism $\iota$ induces an embedding of the product of two modular curves into the Shimura variety $S \times Y$ whose cohomology class generates an $E$-subspace 
$$
\mathcal{Z}(\Pi_f \times \sigma_f) \subset (M_B(\Pi_f \times \sigma_f, 2)_{\R} \cap M^{0, 0})^+.
$$
Let $p(\Pi)$, resp. $p(\sigma)$, be the de Rham-Whittaker periods attached to $\Pi$, resp. $\sigma$, defined and suitably normalized in section \ref{dR-W} and let $p(\Pi \times \sigma)$ denote the product $p(\Pi)p(\sigma)$.

\begin{thm} \label{main2}  Assume that $\Pi$ and $\sigma$ have trivial central characters and that $L(s, \Pi \times \sigma)$ has a pole at $s=1$. Then
$$
\mathcal{Z}(\Pi_f \times \sigma_f)=p(\Pi \times \sigma)\Res_{s=1}L(s, \Pi \times \sigma) \mathcal{D}(\Pi_f \times \sigma_f).
$$
\end{thm}

Three remarks are in order. The first is that the poles of $L(s, \Pi \times \sigma)$ are at most simple (\cite{moriyama1} Thm. 1.1). As a consequence, the residue in Thm. \ref{main2} is nothing but the special value at $1$. The second is that, according to \cite{asgari-shahidi} and \cite{shahidi} Thm. 5.2, we have $L(0, \Pi \times \sigma) \neq 0$. As a consequence it follows from the second point of Beilinson's conjecture that the integral motivic cohomology space corresponding to $\Pi \times \sigma$ is zero. This explains why we only need one cycle class in the above theorem. The third remark is that according to Beilinson's conjecture for $\Res_{s=1}L(s, \Pi \times \sigma)$, we should have $p(\Pi \times \sigma) \in E^\times$. We hope to address this problem in a future work. Let us also point out the following straightforward corollary of the theorem, in the spirit of the conjecture of Tate.

\begin{cor} \label{main1} Assume that $\Pi$ and $\sigma$ have trivial central characters and that $L(s, \Pi \times \sigma)$ has a pole at $s=1$. Then
$
\mathcal{Z}(\Pi_f \times \sigma_f) \neq 0$.
\end{cor}

Let us briefly outline the contents of the paper. In section \ref{cohshvar}, we review what is needed about the de Rham and Betti cohomologies of the Shimura varieties $S$ and $Y$. In section \ref{sect23}, we recall the definition of the Deligne rational structure $\mathcal{D}(\Pi_f \times \sigma_f)$ and give the precise definition of $\mathcal{Z}(\Pi_f \times \sigma_f)$. In section \ref{pduality}, we explain the computation of some Poincar\'e duality pairings that turn out to be crucial in the proofs of our main theorems. In section \ref{zetaint} we state and prove results about the integral representation of the $L(s, \Pi \times \sigma)$ and in section \ref{pfs} we prove the two theorems stated above.\\

\textbf{Aknowledgments.}  I would like to aknowledge a funding of the ANR-13-BS01-0012
FERPLAY (Formule des traces relative, p\'eriodes, fonctions $L$ et analyse harmonique). It is a pleasure to thank Eric Urban for his invitation to Columbia University where I started to think about this work and David Loeffler, Nadir Matringe, Tomonori Moriyama, J\"org Wildeshaus and Sarah Zerbes for stimulating conversations and correspondence. Finally, I would like to thank the anonymous referees for their careful reading, for pointing out several inaccuracies and suggesting improvements.

\section{Motives for $\mrm{GSp}(4) \times \mrm{GL}(2)$}

\subsection{Cohomology of Shimura varieties} \label{cohshvar}Let us briefly recall the definition of the Shimura varieties we are interested in. Details can be found in \cite{laumon} Partie I.3. Let $I_2$ be the identity matrix of size two and let $\mathcal{J}$ be the symplectic form on $\Z^4$ whose  matrix in the canonical basis is
$$
\mathcal{J}=
\begin{pmatrix}
 & I_2\\
-I_2 & \\
\end{pmatrix}.
$$
The symplectic group $\mathrm{GSp}(4)$ is defined as
$$
\mrm{GSp}(4)=\left\{g \in \mathrm{GL}(4) \,|\, g^t \mathcal{J} g = \nu(g) \mathcal{J}, \, \nu(g) \in \mathbb{G}_m\right\}.
$$
Then the map $\nu: \mrm{GSp}(4) \rightarrow \mathbb{G}_m$ is a character. Let $\mathbb{S}=Res_{\mathbb{C}/\mathbb{R}} \mathbb{G}_{m, \mathbb{C}}$ be the Deligne torus and let $\mathcal{H}$ be the $\mrm{GSp}(4, \mathbb{R})$-conjugacy class of the morphism $h: \mathbb{S} \rightarrow \mrm{GSp}(4)_\mathbb{R}$ given on $\mathbb{R}$-points by
$$
x+iy \longmapsto \begin{pmatrix}
x & & y & \\
 & x &  & y\\
-y &  & x & \\
 &  -y &  & x\\
\end{pmatrix}.
$$
The pair $(\mrm{GSp}(4), \mathcal{H})$ is a pure Shimura datum in the sense of \cite{pink} 2.1 and $\mathcal{H} \simeq \mathcal{H}^+ \sqcup \mathcal{H}^-$ is isomorphic to the disjoint union of the Siegel upper and lower half space of genus $2$. It is easy to see that the reflex field of $(\mrm{GSp}(4), \mathcal{H})$ is the field of rational numbers. For any neat compact open subgroup $K$ of $\mrm{GSp}(4, \A_f)$, let us denote by $S^K$ the Shimura variety at level $K$ associated to $(\mrm{GSp}(4), \mathcal{H})$. It is a smooth quasi-projective variety over $\Q$ such that, as complex analytic varieties, we have
$$
S^K(\C) = \mrm{GSp}(4, \Q) \backslash (\mathcal{H} \times \mrm{GSp}(4, \A_f)/K).
$$
If $K=K(N) \subset \mrm{GSp}(4, \hat{\Z})$ is the principal congruence subgroup of level $N \geq 3$, then
$$
S^{K(N)}(\C)\simeq \bigsqcup_{(\Z/N\Z)^\times} \Gamma(N) \backslash \mathcal{H}^+
$$
where $\Gamma(N) \subset \mrm{Sp}(4, \Z)$ is the principal congruence subgroup of level $N$. In particular, the complex analytic varieties $S^K(\C)$ have dimension $3$. For $g \in \mrm{GSp}(4, \mathbb{A}_f)$ and $K$, $K'$ two neat compact open subgroups of $\mrm{GSp}(4,\mathbb{A}_f)$ such that $g^{-1}K'g \subset K$, right multiplication by $g$ on $S^K(\C)$ descends to a morphism $[g]: S^{K'} \rightarrow S^K$ of $\mathbb{Q}$-schemes, which is finite and \'etale. This implies that there is an action of $\mrm{GSp}(4, \mathbb{A}_f)$ on the projective system $(S^K)$ indexed by neat compact open subgroups of $\mrm{GSp}(4, \mathbb{A}_f)$. In what follows, all compact open subgroups of $\mrm{GSp}(4, \mathbb{A}_f)$ and of $\mrm{GL}(2, \A_f)$ will be assumed to be neat and we will not mention this fact anymore. Similarly, we have the projective system of modular curves $(Y^L)$ indexed by the set of compact open subgroups $L \subset \mrm{GL}(2, \A_f)$ and which is endowed of an action of $\mrm{GL}(2, \A_f)$. Let $G$ denote the product $G=\mrm{GSp}(4) \times \mrm{GL}(2)$. This discussion shows that we have an action of $G(\A_f)$ on the projective system $(S^K \times Y^L)$ indexed by the set of pairs $(K, L)$ where $K$ is a compact open subgroup of $\mrm{GSp}(4, \mathbb{A}_f)$ and $L$ is a compact open subgroup of $\mrm{GL}(2, \A_f)$.\\

Let $H^*_{dR, c}(S^K \times Y^L)$ and $H^*_{dR}(S^K \times Y^L)$ be the de Rham cohomology with compact support and without support respectively, with complex coefficients. Let
$$
H^*_{dR, !}(S^K \times Y^L)=\im(H^*_{dR, c}(S^K \times Y^L) \rightarrow H^*_{dR}(S^K \times Y^L))
$$ 
and let us define $H^*_{dR, !}(S^K)$ and $H^*_{dR, !}(Y^L)$ similarly. 

\begin{lem} \label{interior-coh} The graded $\C$-vector space $H^*_{dR, !}(Y^L)$ vanishes outside degree $1$.
\end{lem}

\begin{proof} Note that $H^0_{dR, c}(Y^L)=0$ because the connected components of $Y^L$ are non-compact. As a consequence $H^2_{dR}(Y^L)=0$ by Poincar\'e duality. This implies the statement.
\end{proof}

Let 
$$
H^4_{dR, !}(S \times Y)=\underrightarrow{\lim}_K \underrightarrow{\lim}_L H^4_{dR, !}(S^K \times Y^L)
$$
where the limit is indexed by the compact open supbgroups $K$ of $\mrm{GSp}(4, \A_f)$ and $L$ of $\mrm{GL}(2, \A_f)$. The action of $G(\A_f)$ on the projective system $(S^K \times Y^L)$ induces a structure of $\C[G(\A_f)]$-module on $H^4_{dR, !}(S \times Y)$. Let us define $H^3_{dR, !}(S)$ and $H^1_{dR, !}(Y)$ similarly as $H^4_{dR, !}(S \times Y)$. By the K\"unneth formula and Lem. \ref{interior-coh}, we have a canonical isomorphism of $\C[G(\A_f)]$-modules
\begin{equation} \label{kunneth}
H^4_{dR, !}(S \times Y) \simeq H^3_{dR, !}(S) \otimes H^1_{dR, !}(Y).
\end{equation}
To relate $H^4_{dR, !}(S \times Y)$ to automorphic representations, let 
$
P_4=\{ \Pi_\infty^{3,0}, \Pi_\infty^{2,1}, {\Pi}_\infty^{1,2}, {\Pi}_\infty^{0,3} \}
$ 
denote the set of isomorphism classes of discrete series of $\mrm{GSp}(4, \R)_+=\nu^{-1}(\R^\times_+)$ with the same central and infinitesimal characters as the trivial representation. Here $\Pi_\infty^{3,0}$ is holomorphic, $\Pi_\infty^{2,1}$ and $\Pi_\infty^{1,2}$ are generic, which means that they have a Whittaker model in the sense of section \ref{section-zeta}, and are determined by their Hodge types as indicated by the formulae displayed on p.8 and finally $\Pi_\infty^{0,3}$ is antiholomorphic. The fact that $P_4$ has $4$ elements follows from Harish-Chandra's classification, as explained in \cite{lemma2} Prop. 3.1 in the particular case where $k=k'=0$, with the notation of loc. cit. In what follows, we will denote by $\Pi_\infty^H$ and $\Pi_\infty^W$ the discrete series of $\mrm{GSp}(4, \R)$ defined as  
\begin{eqnarray*}
\Pi_\infty^H=\mrm{Ind}_{\mrm{GSp}(4, \R)_+}^{\mrm{GSp}(4, \R)} \Pi_\infty^{3,0}=\mrm{Ind}_{\mrm{GSp}(4, \R)_+}^{\mrm{GSp}(4, \R)} \Pi_\infty^{0,3},\\
\Pi_\infty^W=\mrm{Ind}_{\mrm{GSp}(4, \R)_+}^{\mrm{GSp}(4, \R)} \Pi_\infty^{2,1}=\mrm{Ind}_{\mrm{GSp}(4, \R)_+}^{\mrm{GSp}(4, \R)} \Pi_\infty^{1,2}.
\end{eqnarray*}
In particular, we have
\begin{eqnarray*}
\Pi_\infty^H|_{\mrm{GSp}(4, \R)_+}= \Pi_\infty^{3,0} \oplus \Pi_\infty^{0,3},\\
\Pi_\infty^W|_{\mrm{GSp}(4, \R)_+}= \Pi_\infty^{2,1} \oplus\Pi_\infty^{1,2}.
\end{eqnarray*}
Similarly, let
$
P_2=\{\sigma_\infty^{1,0}, \sigma_\infty^{0,1}\}
$
denote the set of isomorphism classes of discrete series of $\mrm{GL}(2, \R)_+=\det^{-1}(\R^\times_+)$ with the same central and infinitesimal characters as the trivial representation, where $\sigma_\infty^{1,0}$ is holomorphic and $\sigma_\infty^{0,1}$ is antiholomorphic. Let $\sigma_\infty$ denote the induced representation 
$$
\sigma_\infty=\mrm{Ind}_{\mrm{GL}(2, \R)_+}^{\mrm{GL}(2, \R)} \sigma_\infty^{1,0}=\mrm{Ind}_{\mrm{GL}(2, \R)_+}^{\mrm{GL}(2, \R)} \sigma_\infty^{0,1}
$$
so that $\sigma_\infty|_{\mrm{GL}(2, \R)_+}=\sigma_\infty^{1,0} \oplus \sigma_\infty^{0,1}$. Let $\mathfrak{gsp}_4$ and $\mathfrak{gl}_2$ denote the complex Lie algebras of $\mrm{GSp}(4)$ and $\mrm{GL}(2)$ respectively and let $K_\infty$ and $L_\infty$ denote respectively the group $\R_+^\times \mrm{U}(2, \R)$, regarded as a maximal compact modulo the center subgroup of $\mrm{GSp}(4, \R)_+$, and the group $\R_+^\times \mrm{SO}(2, \R)$, regarded as a maximal compact modulo the center subgroup of $\mrm{GL}(2, \R)_+$, in the standard way. Using the relative Lie algebra cohomology groups, we have the following result.

\begin{pro} \label{decomposition} There is a canonical $G(\A_f)$-equivariant isomorphism
\begin{equation*}
H^4_{dR, !}(S \times Y)\simeq \bigoplus_{\Pi, \sigma} H^3(\mathfrak{gsp}_{4}, K_\infty, \Pi_\infty)^{m(\Pi)} \otimes H^1(\mathfrak{gl}_{2}, L_\infty, \sigma_\infty)^{m(\sigma)} \otimes (\Pi_f \times \sigma_f)
\end{equation*}
where the direct sum is indexed by equivalence classes of cuspidal automorphic representations $\Pi=\Pi_\infty \otimes \Pi_f$ of $\mrm{GSp}(4, \A)$ and $\sigma=\sigma_\infty \otimes \sigma_f$ of $\mrm{GL}(2, \A)$ such that $\Pi_\infty|_{\mrm{GSp}(4, \R)_+} \in P_4$ and $\sigma_\infty |_{\mrm{GL}(2, \R)_+} \in P_2$ and where $m(\Pi)$ and $m(\sigma)$ denote the cuspidal multiplicities of $\Pi$ and $\sigma$ respectively. 
\end{pro}

\begin{proof} This follows from (\ref{kunneth}) and the analogous statements for $ H^3_{dR, !}(S)$, which follows from \cite{lemma2} (8) and (9), and for $H^1_{dR, !}(Y)$, which follows from \cite{scholze} Lem. 12.3 once noticed that for the smooth curves $Y^L$, one has a canonical isomorphism $H^1_{dR, !}(Y^L) \simeq H^1_{(2)}(Y^L)$, where $H^1_{(2)}(Y^L)$ denotes $L^2$-cohomology.
\end{proof}

Let $\mathfrak{k}$ denote the complex Lie algebra of $K_\infty$ and $\mathfrak{l}$ denote the complex Lie algebra of $L_\infty$.

\begin{pro} \label{dimension} For any $\Pi_\infty^{p,q} \in P_4$, resp. $\sigma_\infty^{r,s} \in P_2$, we have
\begin{eqnarray*}
H^3(\mathfrak{gsp}_{4}, K_\infty, \Pi_\infty^{p,q}) &=&\Hom_{K_\infty}\left( \bigwedge^3  \mathfrak{gsp}_{4}/\mathfrak{k}, \Pi_\infty^{p,q} \right),\\
H^1(\mathfrak{gl}_{2}, L_\infty, \sigma_\infty^{r,s}) &=&\Hom_{L_\infty}\left(\mathfrak{gl}_{2}/\mathfrak{l},\sigma_\infty^{r,s} \right)
\end{eqnarray*}
and these $\C$-vector spaces are one-dimensional. 
\end{pro}

\begin{proof} According to \cite{borel-wallach} II. \S 3, Prop. 3.1, for any $\Pi_\infty^{p,q} \in P_4$, resp. $\sigma_\infty^{r,s} \in P_2$, the $(\mathfrak{gsp}_{4}, K_\infty)$-complex of $\Pi_\infty^{p,q}$, resp. the $(\mathfrak{gl}_{2}, L_\infty)$-complex of $\sigma_\infty^{r,s}$, has zero differential. This implies the two equalities in the statement of the Proposition. The statement about the dimensions is a particular case of loc. cit. II. Thm. 5.3.
\end{proof}

Let us fix a cuspidal automorphic representation $\Pi=\Pi_\infty \otimes \Pi_f$, resp. $\sigma=\sigma_\infty \otimes \sigma_f$, of $\mrm{GSp}(4, \A)$, resp. $\mrm{GL}(2, \A)$. Note that, by definition of the relative Lie algebra cohomology complex 
$$
H^3(\mathfrak{gsp}_{4}, K_\infty, \Pi_\infty)=H^3(\mathfrak{gsp}_{4}, K_\infty, \Pi_\infty|_{\mrm{GSp}(4, \R)_+}).
$$
Assume that $\Pi_\infty|_{\mrm{GSp}(4, \R)_+} \in P_4$, resp. $\sigma_\infty |_{\mrm{GL}(2, \R)_+} \in P_2$. As the non-archimedean part $\Pi_f$ of $\Pi$ contributes to coherent cohomology,  it is defined over a number field (see \cite{bhr} for more details). The fact that the analogous statement for $\sigma_f$ is true is proved in \cite{waldspurger}. Let us denote by $E$ a fixed number field over which both $\Pi_f$ and $\sigma_f$ are defined. Let $H^4_{B,!}(S \times Y, \Q)$ denote the image of the Betti cohomology with compact support in the Betti cohomology without support, with rational coefficients. We will use the following notation
\begin{eqnarray*}
M_B(\Pi_f \times \sigma_f) &=& \mrm{Hom}_{E[G(\A_f)]}(\Pi_f \times \sigma_f, H^4_{B, !}(S \times Y, \Q) \otimes_{\Q} E).
\end{eqnarray*}
In the next result, we need to assume that $\Pi$ is globally generic. Let us recall that this means that there exists $\Psi \in \Pi$ such that the function $W_\Psi: \mrm{GSp}(4, \A) \rightarrow \C$ defined by (\ref{global-whittaker}) in section \ref{section-zeta} is not identically zero.

\begin{pro} \label{hyp-CAP} Assume that $\Pi$ is globally generic. Then, the natural inclusion of inner cohomology in usual cohomology $H^4_{B, !}(S \times Y, \Q) \subset H^4_{B}(S \times Y, \Q)$ induces an isomorphism
$$
M_B(\Pi_f \times \sigma_f) = \mrm{Hom}_{E[G(\A_f)]}(\Pi_f \times \sigma_f, H^4_{B}(S \times Y, \Q) \otimes_{\Q} E).
$$
\end{pro}

\begin{proof} As $\Pi$ is globally generic, it is not CAP according to \cite{pssoudry2} Thm. 1.1. Hence, we can apply \cite{weissauer2} Thm. 1.1 which implies that 
\begin{eqnarray*}
\mrm{Hom}_{E[\mrm{GSp}(4, \A_f)]}(\Pi_f, H^3_{B}(S, \Q) \otimes_{\Q} E) &=& \mrm{Hom}_{E[\mrm{GSp}(4, \A_f)]}(\Pi_f, H^3_{B,!}(S, \Q) \otimes_{\Q} E)\\
\mrm{Hom}_{E[\mrm{GSp}(4, \A_f)]}(\Pi_f, H^4_{B}(S, \Q) \otimes_{\Q} E) &=& 0.
\end{eqnarray*}
Note that  \cite{weissauer2} Thm. 1.1 assumes that $\Pi_\infty \simeq \Pi_\infty^H$, but that this fact is not used in the proof and so the same proof is true when $\Pi_\infty \simeq \Pi_\infty^W$. Furthermore, as there are no CAP representations for $\mrm{GL}(2)$ the same argument as in the proof of  \cite{weissauer2} Thm. 1.1 shows that
$$
\mrm{Hom}_{E[\mrm{GL}(2, \A_f)]}(\sigma_f, H^1_{B}(Y, \Q) \otimes_{\Q} E) = \mrm{Hom}_{E[\mrm{GL}(2, \A_f)]}(\sigma_f, H^1_{B,!}(Y, \Q) \otimes_{\Q} E).
$$
Hence, by the K\"unneth formula, we have 
$$
\mrm{Hom}_{E[G(\A_f)]}(\Pi_f \times \sigma_f, H^4_{B}(S \times Y, \Q) \otimes_{\Q} E)=M_B(\Pi_f \times \sigma_f).
$$
\end{proof}

Let $M_B(\Pi_f)$ and $M_B(\sigma_f)$ be defined similarly as $M_B(\Pi_f \times \sigma_f)$ via the Betti cohomology of $S$ and $Y$ respectively. Then, the K\"unneth formula implies that 
$$
M_B(\Pi_f \times \sigma_f)=M_B(\Pi_f) \otimes_{E} M_B(\sigma_f).
$$
According to Prop. \ref{decomposition}, Prop. \ref{dimension} and the comparison isomorphism between de Rham and Betti cohomology, these are finite dimensional $\Q$-vector spaces endowed with a $\Q$-linear action of $E$ and additional structures as follows. Let $M_B(\Pi_f)_{\C}$ and $M_B(\sigma_f)_{\C}$ denote the vector spaces obtained after extending the scalars from $\Q$ to $\C$. Then we have the Hodge decompositions
\begin{eqnarray} \label{hodge}
M_B(\Pi_f)_{\C} &=& M(\Pi_f)^{3, 0} \oplus M(\Pi_f)^{2, 1} \oplus M(\Pi_f)^{1, 2} \oplus M(\Pi_f)^{0, 3},\\
M_B(\sigma_f)_{\C}&=& M(\sigma_f)^{1, 0} \oplus M(\sigma_f)^{0, 1}
\label{hodge2}
\end{eqnarray}
where
\begin{eqnarray*}
M(\Pi_f)^{3, 0} &=&  \bigoplus_{\sigma: E \rightarrow \C} H^3(\mathfrak{gsp}_{4}, K_\infty, \Pi_\infty^{3,0})^{m(\Pi_\infty^H \otimes \Pi_f)},\\
M(\Pi_f)^{2, 1} &=& \bigoplus_{\sigma: E \rightarrow \C} H^3(\mathfrak{gsp}_{4}, K_\infty, \Pi_\infty^{2,1})^{m(\Pi_\infty^W \otimes \Pi_f)}, \\
M(\Pi_f)^{1, 2} &=& \bigoplus_{\sigma: E \rightarrow \C} H^3(\mathfrak{gsp}_{4}, K_\infty, {\Pi}_\infty^{1,2})^{m(\Pi_\infty^W \otimes \Pi_f)},\\
M(\Pi_f)^{0, 3} &=& \bigoplus_{\sigma: E \rightarrow \C} H^3(\mathfrak{gsp}_{4}, K_\infty, {\Pi}_\infty^{0,3})^{m(\Pi_\infty^H \otimes \Pi_f)},\\
M(\sigma_f)^{1,0} &=& \bigoplus_{\sigma: E \rightarrow \C} H^1(\mathfrak{gl}_{2}, L_\infty, {\sigma}_\infty^{1,0})^{m({\sigma}_\infty \otimes \sigma_f)},\\
M(\sigma_f)^{0,1} &=& \bigoplus_{\sigma: E \rightarrow \C} H^1(\mathfrak{gl}_{2}, L_\infty, \sigma_\infty^{0,1})^{m(\sigma_\infty \otimes \sigma_f)}.
\end{eqnarray*}
Furthermore $M_B(\Pi_f \times \sigma_f)$ has the tensor product Hodge structure
\begin{equation} \label{hodge-dec}
M_B(\Pi_f \times \sigma_f)_{\C} = M^{4,0} \oplus M^{3,1} \oplus M^{2,2} \oplus M^{1,3} \oplus M^{0,4}.
\end{equation}
The following definition is taken from \cite{beilinson2} \S 7.

\begin{defn} \label{realmixed} Let $A$ be a subring of $\R$. A real mixed $A$-Hodge structure is a mixed $A$-Hodge structure whose underlying $A$-vector space is endowed with an involution $F_\infty$ stabilizing the weight filtration and whose $\mbb{C}$-antilinear complexification $\overline{F}_\infty$ stabilizes the Hodge filtration. 
\end{defn}

Let $\mathrm{MHS}_{A}^+$ denote the abelian category of real mixed $A$-Hodge structures.

\begin{defn} \label{mhs-coeff} Let $F$ be a ring and let $A$ be a subring of $\R$. A real mixed $A$-Hodge structure with coefficients in $F$ is a pair $(M, s)$ where $M$ is an object of $\mathrm{MHS}_{A}^+$  and $s: F \lra \mrm{End}_{\mathrm{MHS}_A^+}(M)$ is a ring homomorphism.
\end{defn}

For any ring $F$, let $\mathrm{MHS}_{\R, F}^+$ denote the abelian category of real mixed $\R$-Hodge structures with coefficients in $F$. The proof of the following result is straightforward.

\begin{pro} \label{hodge-deRham-pif} Let $F_\infty$ be the involution on $M_B(\Pi_f \times \sigma_f)$ induced by the complex conjugation on $S(\mbb{C}) \times Y(\C)$. Then $(M_B(\Pi_f \times \sigma_f), F_\infty)$ is an object of $\mathrm{MHS}_{\mbb{Q}, E}^+$ which is pure of weight $4$.
\end{pro}

Furthermore, it follows from \cite{harris1} Cor. 2.3.1 that there exists a filtered $E$-vector space $(M_{dR}(\Pi_f \times \sigma_f), F^* M_{dR}(\Pi_f \times \sigma_f))$ and a comparison isomorphism
\begin{equation} \label{comparison}
I_\infty: M_{B}(\Pi_f \times \sigma_f)_{\C} \rightarrow M_{dR}(\Pi_f \times \sigma_f)_{\C}
\end{equation}
such that the Hodge filtration of $M_B(\Pi_f \times \sigma_f)$, defined as
$$
F^p_{\C}=\bigoplus_{p' \geq p} M^{p', q},
$$
satisfies $I_\infty(F^p_{\C})=F^p  M_{dR}(\Pi_f \times \sigma_f)_{\C}$. A similar statement holds for $M_B(\Pi_f)$ and $M_B(\sigma_f)$.

\subsection{The Deligne rational structure and the cycle class} \label{sect23}

\subsubsection{The Deligne rational structure} Let $\Pi=\Pi_\infty \otimes \Pi_f$ and $\sigma=\sigma_\infty \otimes \sigma_f$ be irreducible cuspidal automorphic representations of $\mrm{GSp}(4, \A)$ and $\mrm{GL}(2, \A)$ respectively. Assume that $\Pi_\infty|_{\mrm{GSp}(4, \R)_+} \in P_4$ and $\sigma_\infty|_{\mrm{GL}(2, \R)_+} \in P_2$. For any integer $n$, let $M_B(\Pi_f \times \sigma_f, n)$ denote the object of $\mathrm{MHS}_{\mbb{Q}, E}^+$ defined as 
$$
M_B(\Pi_f \times \sigma_f, n)=M_B(\Pi_f \times \sigma_f) \otimes_{\Q} \Q(n)
$$ 
where $\Q(n)$ is the $n$-th tensor power of the Tate object. Let $M_B(\Pi_f \times \sigma_f, n)^\pm$ denote the subspace of $M_B(\Pi_f \times \sigma_f, n)$ where $F_\infty$ acts as $\pm 1$. The comparison isomorphism $I_\infty^{-1}$ (see (\ref{comparison}))
between de Rham and Betti cohomology, sends the real structure $M_{dR}(\Pi_f \times \sigma_f)_{\R}$ of $M_{dR}(\Pi_f \times \sigma_f)_{\C}$ to the real structure $M_B(\Pi_f \times \sigma_f)_{\R}^+ \oplus M_B(\Pi_f \times \sigma_f)_{\R}^-(-1)$ of $M_{B}(\Pi_f \times \sigma_f)_{\C}$, where $M_B(\Pi_f \times \sigma_f)_{\R}^-(-1)$ simply denotes the sub-$\R \otimes_{\Q} E$-module $M_B(\Pi_f \times \sigma_f)_{\R}^- \otimes i\R$ of $M_B(\Pi_f \times \sigma_f)_{\C}^-$. In particular, we have a natural $\R \otimes_{\Q} E$-linear map $$
F^3 M_{dR}(\Pi_f \times \sigma_f)_{\R} \rightarrow M_B(\Pi_f \times \sigma_f)_{\R}^+
$$
defined as the composition of the natural inclusion $F^3 M_{dR}(\Pi_f \times \sigma_f)_{\R} \subset M_{dR}(\Pi_f \times \sigma_f)_{\R}$, of $I_\infty^{-1}$ and of the natural projection $M_B(\Pi_f \times \sigma_f)_{\R}^+ \oplus M_B(\Pi_f \times \sigma_f)_{\R}^-(-1) \rightarrow M_B(\Pi_f \times \sigma_f)_{\R}^+$. Composing with the canonical isomorphism $M_B(\Pi_f \times \sigma_f)_{\R}^+ \simeq M_B(\Pi_f \times \sigma_f, 2)_{\R}^+$ given by multiplication by $(2\pi i)^2$, we obtain the natural map $$F^3 M_{dR}(\Pi_f \times \sigma_f)_{\R} \rightarrow M_B(\Pi_f \times \sigma_f, 2)_{\R}^+.$$

\begin{pro} \label{shortexact} We have the following canonical short exact sequence of $\R \otimes_{\Q} E$-modules 
$$
0 \rightarrow F^3 M_{dR}(\Pi_f \times \sigma_f)_{\R} \rightarrow M_B(\Pi_f \times \sigma_f, 2)_{\R}^+ \rightarrow \mrm{Ext}^1_{\mrm{MHS}_{\R}^+}(\R(0), M_B(\Pi_f \times \sigma_f, 3)_{\R}) \rightarrow 0
$$
where the second map is the map defined above.
\end{pro}

\begin{proof}
As $M_B(\Pi_f \times \sigma_f, 3)$ is pure of weight $-2$, it follows for example from \cite{lemma2} Lem. 4.11 (see also \cite{nekovar} sections 2.2 and 2.3) that we have the canonical short exact sequence
$$
0 \rightarrow F^0 M_{dR}(\Pi_f \times \sigma_f, 3)_{\R} \rightarrow M_B(\Pi_f \times \sigma_f, 3)_{\R}^-(-1) \rightarrow \mrm{Ext}^1_{\mrm{MHS}_{\R}^+}(\R(0), M_B(\Pi_f \times \sigma_f, 3)_{\R}) \rightarrow 0
$$
where the second map is defined similarly as above. Furthermore, we have the following canonical isomorphisms $F^3 M_{dR}(\Pi_f \times \sigma_f)_{\R} \simeq F^0 M_{dR}(\Pi_f \times \sigma_f, 3)_{\R}$ and $M_B(\Pi_f \times \sigma_f, 2)_{\R}^+ \simeq M_B(\Pi_f \times \sigma_f, 3)_{\R}^-(-1)$. The conclusion follows.
\end{proof}

\begin{rems} Let us denote by $(M_B(\Pi_f \times \sigma_f, 2)_{\R} \cap M^{0, 0})^+$ the vectors of $M_B(\Pi_f \times \sigma_f, 2)_{\R}^+$ which have Hodge type $(0, 0)$. It is straightforward to deduce from the previous result a canonical isomorphism
$$
\mrm{Ext}^1_{\mrm{MHS}_{\R}^+}(\R(0), M_B(\Pi_f \times \sigma_f, 3)_{\R}) \simeq (M_B(\Pi_f \times \sigma_f, 2)_{\R} \cap M^{0, 0})^+
$$
as claimed in the introduction. Note also that $
\mrm{Ext}^1_{\mrm{MHS}_{\R}^+}(\R(0), M_B(\Pi_f \times \sigma_f, 3)_{\R})$ can be thought as the $(\Pi_f \times \sigma_f)$-isotypic component of a Deligne-Beilinson cohomology space.
\end{rems}

\begin{pro} \label{ranks} The ranks of the $\R \otimes_{\Q} E$-modules $F^3 M_{dR}(\Pi_f \times \sigma_f)_{\R}$, $M_B(\Pi_f \times \sigma_f, 2)_{\R}^+$ and $\mrm{Ext}^1_{\mrm{MHS}_{\R}^+}(\R(0), M_B(\Pi_f \times \sigma_f, 3)_{\R})$ are finite and equal to $2m(\Pi_\infty^H \otimes \Pi_f)+m(\Pi_\infty^W \otimes \Pi_f)$, $2m(\Pi_\infty^H \otimes \Pi_f)+2m(\Pi_\infty^W \otimes \Pi_f)$ and $m(\Pi_\infty^W \otimes \Pi_f)$ respectively.
\end{pro}

\begin{proof} The existence of the short exact sequence of Prop. \ref{shortexact} implies that it is enough to prove
\begin{eqnarray*}
rk \, F^3 M_{dR}(\Pi_f \times \sigma_f)_{\R} &=& 2m(\Pi_\infty^H \otimes \Pi_f)+m(\Pi_\infty^W \otimes \Pi_f),\\
rk \, M_B(\Pi_f \times \sigma_f, 2)_{\R}^+ &=& 2m(\Pi_\infty^H \otimes \Pi_f)+2m(\Pi_\infty^W \otimes \Pi_f).
\end{eqnarray*}
According to \cite{shalika} Thm. 5.5, we have $m(\sigma_\infty \otimes \sigma_f)=1$. Hence both statements are direct consequences of Prop. \ref{dimension} and equality (\ref{hodge-dec}).
\end{proof}

The short exact sequence of Prop. \ref{shortexact} induces the canonical isomorphism of rank one $\R \otimes_{\Q} E$-modules
$$
\det\nolimits_{\R \otimes_{\Q} E} F^3 M_{dR}(\Pi_f \times \sigma_f)_{\R} \otimes_{\R \otimes_{\Q} E} \det\nolimits_{\R \otimes_{\Q} E} \mrm{Ext}^1_{\mrm{MHS}_{\R}^+}(\R(0), M_B(\Pi_f \times \sigma_f, 3)_{\R})$$ $$ \simeq \det\nolimits_{\R \otimes_{\Q} E} M_B(\Pi_f \times \sigma_f, 2)_{\R}^+,
$$
where we denote by $\det$ the highest exterior power. Let $\det\nolimits_E F^3 M_{dR}(\Pi_f \times \sigma_f)^\vee$ denote the $E$-module dual of $\det_E F^3 M_{dR}(\Pi_f \times \sigma_f)$. The evaluation map $$\det\nolimits_E F^3 M_{dR}(\Pi_f \times \sigma_f) \otimes_E \det\nolimits_E F^3 M_{dR}(\Pi_f \times \sigma_f)^\vee \rightarrow E$$ is an isomorphism. The following definition is taken from \cite{den-scholl} (2.3.2).

\begin{defn} \label{beilinson-qs} The Beilinson $E$-structure on $\det_{\R \otimes_{\Q} E} \mrm{Ext}^1_{\mrm{MHS}_{\R}^+}(\R(0), M_B(\Pi_f \times \sigma_f, 3)_{\R})$ is defined as 
$$
\mathcal{B}(\Pi_f \times \sigma_f)=\det\nolimits_E F^3 M_{dR}(\Pi_f \times \sigma_f)^\vee \otimes_E \det\nolimits_E M_B(\Pi_f \times \sigma_f, 2)^+.
$$ 
\end{defn}

Let us denote by $\delta(\Pi_f \times \sigma_f, 3)$ the determinant of the comparison isomorphism 
$$
M_B(\Pi_f \times \sigma_f, 3)_{\C} \rightarrow M_{dR}(\Pi_f \times \sigma_f, 3)_{\C}
$$
computed in basis defined on $E$ on both sides. Then $\delta(\Pi_f \times \sigma_f, 3)$ is an element of $(\C \otimes E)^\times$, which, as $\dim_E M_B(\Pi_f \times \sigma_f)^-$ is even, belongs to $(\R \otimes_{\Q} E)^\times$ (see \cite{valeurs-deligne} p. 320) and which is independent of the choice of the basis up to right multiplication by an element of $E^\times$. 

\begin{defn} \label{deligne-qs} The Deligne $E$-structure on $\det_{\R \otimes_{\Q} E} \mrm{Ext}^1_{\mrm{MHS}_{\R}^+}(\R(0), M_B(\Pi_f \times \sigma_f, 3)_{\R})$ is defined as 
\begin{eqnarray*}
\mathcal{D}(\Pi_f \times \sigma_f) &=& (2\pi i)^{\dim_{E}M_B(\Pi_f \times \sigma_f, 3)^-} \delta(\Pi_f \times \sigma_f, 3)^{-1} \mathcal{B}(\Pi_f \times \sigma_f).
\end{eqnarray*}
\end{defn}

\subsubsection{The cycle class} \label{section-cycle} The reductive group $G$ contains 
$$
H=\mathrm{GL}(2) \times_{\mathbb{G}_m} \mathrm{GL}(2)=\{(g, h) \in \mrm{GL}(2) \times \mrm{GL}(2) \,|\, \det(g)=\det(h) \}
$$
as a closed subgroup via the embedding $\iota: H \rightarrow G$ defined by
$$
\iota\left( \begin{pmatrix}
a & b\\
c & d\\
\end{pmatrix}, \begin{pmatrix}
a' & b'\\
c' & d'\\
\end{pmatrix} \right) =
\left( \begin{pmatrix}
a & & b & \\
 & a' &  & b'\\
c &  & d & \\
 & c' &  & d'\\
\end{pmatrix},  \begin{pmatrix}
a' & b'\\
c' & d'\\
\end{pmatrix} \right).
$$
For any integer $N \geq 3$, this group homomorphism induces the closed embedding $$
\begin{CD}
Y(N) \times_{\Q(\mu_N)} Y(N) @>\iota_N >> S(N) \times Y(N)
\end{CD}$$
between the corresponding Shimura varieties. Let us fix an integer $N$ and let 
\begin{equation} \label{ZN}
\mathcal{Z}_N \in H^4_B(S(N) \times Y(N), \Q(2))
\end{equation}
be the cohomology class of the image of $\iota_N$, which we regard as an element of the $E[G(\A_f)]$-module $H^4_{B}(S \times Y, \Q(2)) \otimes_{\Q} E=\underrightarrow{\lim}_M H^4_{B}(S(M) \times Y(M), \Q(2)) \otimes_{\Q} E$. Let 
$$
\widetilde{\mathcal{Z}}(\Pi_f \times \sigma_f)=\mrm{Hom}_{E[G(\A_f)]}(\Pi_f \times \sigma_f, E[G(\A_f)]\mathcal{Z}_N).
$$
This is a sub-$E$-vector space of $\mrm{Hom}_{E[G(\A_f)]}(\Pi_f \times \sigma_f, H^4_{B}(S \times Y, \Q(2)) \otimes_{\Q} E)$. As we assume that $\Pi$ is globally generic, the latter coincides with $M_B(\Pi_f \times \sigma_f, 2)$ (see Prop. \ref{hyp-CAP}). Note that as the cycle $Y(N) \times_{\Q(\mu_N)} Y(N)$ is defined over $\Q$, we have 
$
\widetilde{\mathcal{Z}}(\Pi_f \times \sigma_f) \subset M_B(\Pi_f \times \sigma_f, 2)^+.
$
We shall denote again by 
$$
{\mathcal{Z}}(\Pi_f \times \sigma_f) \subset \mrm{Ext}^1_{\mrm{MHS}_{\R}^+}(\R(0), M_B(\Pi_f \times \sigma_f, 3)_{\R})
$$
the sub-$E$-vector space defined as the image of $\widetilde{\mathcal{Z}}(\Pi_f \times \sigma_f)$ by the natural map
$$
M_B(\Pi_f \times \sigma_f, 2)^+ \rightarrow M_B(\Pi_f \times \sigma_f, 2)^+_{\R} \rightarrow \mrm{Ext}^1_{\mrm{MHS}_{\R}^+}(\R(0), M_B(\Pi_f \times \sigma_f, 3)_{\R}).
$$
where the first map is the canonical inclusion and the second is the third map in the short exact sequence of Prop. \ref{shortexact}. 

\subsection{Computation of Poincar\'e duality pairings} \label{pduality} From now on, we fix two irreducible cuspidal automorphic representations $\Pi=\Pi_\infty \otimes \Pi_f$ and $\sigma=\sigma_\infty \otimes \sigma_f$ of $\mrm{GSp}(4, \A_f)$ and $\mrm{GL}(2, \A_f)$ respectively. We assume that $\Pi_\infty|_{\mrm{GSp}(4, \R)_+} \in P_4$, that $\sigma_\infty|_{\mrm{GL}(2, \R)_+} \in P_2$ and that $\Pi$ is globally generic. This implies that $\Pi_\infty \simeq \Pi_\infty^W$. The following result will be useful later.

\begin{pro} \label{js} \cite{jiang-soudry} $$m(\Pi_\infty^W \otimes \Pi_f)=1.$$
\end{pro}

\subsubsection{Cohomological interpretation of a period integral} \label{coho-per} The previous result implies that the $\C \otimes E$-module $M(\Pi_f)^{2,1}$ has rank one. Let us explain how to attach a generator of this module to certain cusp forms in the representation space of $\Pi$. We will freely use some standard results and notations from the section 3.1 of \cite{lemma2}. In particular, if $(k, k') \in \Z^2$ is a pair of integers such that $k \geq k'$, let $\tau_{(k, k')}$ denote the irreducible $\C[K_\infty]$-module of highest weight $(k, k')$ with the conventions of loc. cit. Then, the generic member $\Pi_\infty^{2,1}$ of the discrete series $L$-packet $P_4$ contains with multiplicity one $\tau_{(3, -1)}$ as a minimal $K_\infty$-type (see loc. cit. Prop. 3.1). Furthermore, we have the Cartan decomposition $\mathfrak{gsp}_{4}=\mathfrak{k} \oplus \mathfrak{p}^+ \oplus \mathfrak{p}^-$ where
$$
\mathfrak{p}^\pm = \left\{ \begin{pmatrix}
Z & \pm i Z \\
\pm i Z & -Z\\
\end{pmatrix}, Z^t=Z  \in  \mathfrak{gl}_{2} \right\}.
$$
For each symmetric matrix $Z \in \mathfrak{gl}_{2}$, define the element $p_\pm(Z)$ of $\mathfrak{gsp}_{4}$ by
$$
p_\pm(Z)=\begin{pmatrix}
Z & \pm iZ\\
\pm iZ & -Z\\
\end{pmatrix}.
$$
Let $X_{(\alpha_1, \alpha_2)} \in \mathfrak{gsp}_{4}$ be defined as
$$
X_{\pm(2, 0)}=p_\pm \left(  \begin{pmatrix}
1 & \\
 &  \\
\end{pmatrix}\right), X_{\pm(1, 1)}=p_\pm \left(  \begin{pmatrix}
 & 1\\
1 &  \\
\end{pmatrix}\right), X_{\pm(0, 2)}=p_\pm \left(  \begin{pmatrix}
 & \\
 &  1\\
\end{pmatrix}\right).
$$
It follows from an easy computation that $X_{(\alpha_1, \alpha_2)}$ is a root vector corresponding to the root $(\alpha_1, \alpha_2)$ with the conventions of loc. cit. Recall from Prop. \ref{dimension} that
$$
H^3\left( \mathfrak{gsp}_{4}, K_\infty, \Pi_\infty^{2,1} \right) =\Hom_{K_\infty}\left( \bigwedge^3  \mathfrak{gsp}_{4}/\mathfrak{k}, \Pi_\infty^{2,1} \right).
$$

\begin{lem} \label{diff-form} Let $\Psi_\infty \in \Pi_\infty^{2,1}$ be a highest weight vector of the minimal $K_\infty$-type. Then, there exists a unique element $\Omega_{\Psi_\infty} \in H^3(\mathfrak{gsp}_{4}, K_\infty, \Pi_\infty^{2,1})$ such that
$$
\Omega_{\Psi_\infty}(X_{(2,0)} \wedge X_{(1,1)} \otimes X_{(0, -2)})=\Psi_\infty.
$$
\end{lem}

\begin{proof} We have 
$$
\bigwedge^3  \mathfrak{gsp}_{4}/\mathfrak{k}= \bigoplus_{p+q=3} \bigwedge^p \mathfrak{p}^+ \otimes \bigwedge^q \mathfrak{p}^-.
$$
and by a weight computation, we find  
$$
\bigwedge^2 \mathfrak{p}^+ \otimes_\mbb{C} \mathfrak{p}^- = \tau_{(3, -1)} \oplus \tau_{(2, 0)} \oplus \tau_{(1, 1)}\\
$$
as $\C[K_\infty]$-modules. As a consequence, the exsitence of $\Omega_{\Psi_\infty}$ follows from the fact that the vector $X_{(2,0)} \wedge X_{(1,1)} \otimes X_{(0, -2)}$, as $\Psi_\infty$, is a vector of highest weight $(3, -1)$. Its unicity follows from the fact that, according to Prop. \ref{dimension}, the $\C$-vector space $H^3(\mathfrak{gsp}_{4}, K_\infty, \Pi_\infty^{2,1})$ has dimension one.
\end{proof}

Let $\Psi=\Psi_\infty \otimes \Psi_f$ be a cusp form in the space of $\Pi=\Pi_\infty^W \otimes \Pi_f$. Note that we have an equivalence $\Pi_\infty^W|_{\mrm{GSp}(4, \R)_+} \simeq \Pi_\infty^{2,1} \oplus \Pi_\infty^{1,2}$. Assume that $\Psi_\infty$ is a highest weight vector of the minimal $K_\infty$-type of $\Pi_\infty^{2,1}$ and such that $\Psi_f$ is invariant by the principal level $N$ congruence subgroup of $\mrm{GSp}(4, \widehat{\Z})$. Define
\begin{equation} \label{embedding}
\Omega_\Psi=(\Omega_{\Psi_\infty})_{\sigma: E \rightarrow \C} \otimes \Psi_f \in H^3_{B,!}(S, \C) \otimes_{\Q} E.
\end{equation}
Let 
$$
\omega_{\Psi} \in M_B(\Pi_f)_{\C}=\Hom_{E[\mrm{GSp}(4, \A_f)]}(\Pi_f, H^3_{B,!}(S, \C) \otimes_{\Q} E)
$$ be the unique element sending $\Psi_f$ to $\Omega_\Psi$. Then $\omega_{\Psi} \in M^{2,1}(\Pi_f)$. Let $\overline{\Psi}_\infty \in \Pi_\infty^{1,2}$ be defined as in \cite{lemma2} Rem. 3.2. It follows from loc. cit., Prop 3.13 that $
\overline{\omega}_{\Psi} \in M^{1, 2}(\Pi_f)$
correponds to the element
$$
\Omega_{\overline{\Psi}_\infty} \in H^3 \left( \mathfrak{gsp}_{4}, K_\infty, \Pi_\infty^{1,2} \right) =\Hom_{K_\infty}\left( \bigwedge^3  \mathfrak{gsp}_{4}/\mathfrak{k}, \Pi_\infty^{1,2} \right)
$$
characterized by the identity
$$
\Omega_{\overline{\Psi}_\infty}(X_{(0,-2)} \wedge X_{(-1,-1)} \otimes X_{(2, 0)})=\overline{\Psi}_\infty
$$
in the same way as $\omega_{\Psi}$ correponds to $\Omega_{\Psi_\infty}$. Let us denote by $\overline{\Psi}$ the cusp form $\overline{\Psi}_\infty \otimes \Psi_f$ which belongs to the representation space of $\Pi_\infty^W \otimes \Pi_f$.\\

We have the Cartan decomposition $\mathfrak{gl}_{2}=\mathfrak{l} \oplus \mathfrak{p}'^+ \oplus \mathfrak{p}'^-$ where
$$
\mathfrak{p}'^\pm= \left\{ \begin{pmatrix}
z & \pm i z \\
\pm i z & -z\\
\end{pmatrix}, z \in  \mathfrak{gl}_{1} \right\}.
$$
Let $$v^\pm= \begin{pmatrix}
1 & \pm i  \\
\pm i  & -1\\
\end{pmatrix} \in \mathfrak{p}'^\pm.$$ Let $\Phi=\Phi_\infty \otimes \Phi_f$ be a cusp form in the space of $\sigma=\sigma_\infty \otimes \sigma_f$ such that $\Phi_\infty$ is a generator of the minimal $L_\infty$-type of $\sigma_\infty$ and such that $\Phi_f$ is invariant by the principal level $N$ congruence subgroup of $\mrm{GL}(2, \widehat{\Z})$. Then, as in Lem. \ref{diff-form}, we define the element of $H^1\left(\mathfrak{gl}_2, L_\infty, \sigma_\infty \right)=\Hom_{L_\infty} \left( \mathfrak{gl}_{2} / \mathfrak{l}, \sigma_\infty \right)$ by prescribing that it sends $v^+$ to $\Phi_\infty$. As above, we associate to $\Phi$ the harmonic differential form
$$
\Omega_\Phi=(\Omega_{\Phi_\infty})_{\sigma: E \rightarrow \C} \otimes \Phi_f \in H^1_{B,!}(Y, \C) \otimes_{\Q} E
$$
which gives rise to a generator $\eta_\Phi$ of $M^{1,0}(\sigma_f)$.\\

Let us introduce the Poincar\'e duality pairing
\begin{equation} \label{pdp}
H^4_{B,!}(S \times Y, \Q) \otimes_{\Q} H^4_{B,!}(S \times Y, \Q) \rightarrow \Q(-4).
\end{equation}
This becomes a perfect pairing of pure Hodge structures when restricted to invariants by any neat compact open subgroup of $G(\A_f)$. Furthermore this is a $G(\A_f)$-equivariant map when $\Q(-4)$ is endowed with the trivial action. This statement is easily deduced from the K\"unneth formula and from the analogous statements for $S$, see \cite{taylor} p. 295, and for $Y$, which are similar to loc. cit. Then (\ref{pdp}) induces a perfect pairing 
$$
\begin{CD}
M_B(\Pi_f \times \sigma_f) \otimes_{E} M_B(\check{\Pi}_f \times \check{\sigma}_f) @>\langle\,\,,\,\, \rangle_B>> E(-4)_B
\end{CD}
$$
where $\check{\Pi}_f$ and $\check{\sigma}_f$ denote the representations contragredient to $\Pi_f$ and $\sigma_f$, respectively, and where $E(-4)_B$ denotes the Betti realization of the $(-4)$-th power of the Tate motive with coefficients in $E$. Furthermore  $
\langle\,\,,\,\, \rangle_B$ is a morphism of Hodge structure and has a de Rham analogue
$$
\begin{CD}
M_{dR}(\Pi_f \times \sigma_f) \otimes_E M_{dR}(\check{\Pi}_f \times \check{\sigma}_f) @>\langle\,\,,\,\, \rangle_{dR}>> E(-4)_{dR}.
\end{CD}
$$
We assume that the central characters $\omega_{\Pi}$ and $\omega_{\sigma}$ of $\Pi$ and $\sigma$ are trivial.  Hence, because of the canonical isomorphisms $\check{\Pi}_f \simeq \Pi_f \otimes (\omega_{\Pi_f}\circ \nu)^{-1}$ (\cite{weissauer1} Lem. 1.1) and $\check{\sigma}_f \simeq \sigma_f \otimes (\omega_{\sigma_f} \circ \det)^{-1}$  (\cite{jacquet-langlands} Thm. 2.18 (i)), the pairings $
\langle\,\,,\,\, \rangle_B$ and $
\langle\,\,,\,\, \rangle_{dR}$ can be regarded as a pairings
$$
\begin{CD}
M_B(\Pi_f \times \sigma_f) \otimes_E M_B({\Pi}_f \times {\sigma}_f) @>\langle\,\,,\,\, \rangle_B>> E(-4)_B\\
M_{dR}(\Pi_f \times \sigma_f) \otimes_E M_{dR}({\Pi}_f \times {\sigma}_f) @>\langle\,\,,\,\, \rangle_{dR}>> E(-4)_{dR}
\end{CD}
$$
whose complexifications are part of the commutative diagram 
\begin{equation} \label{vertical}
\begin{CD}
M_B(\Pi_f \times \sigma_f)_{\C} \otimes_{\C \otimes_{\Q} E} M_B({\Pi}_f \times {\sigma}_f)_{\C} @>\langle\,\,,\,\, \rangle_{B, \C}>> \C \otimes_{\Q} E(-4)_B\\
@VVV                                                                                                                                 @VVV\\
M_{dR}(\Pi_f \times \sigma_f)_{\C} \otimes_{\C \otimes_{\Q} E} M_{dR}({\Pi}_f \times {\sigma}_f)_{\C} @>\langle\,\,,\,\, \rangle_{dR, \C}>> \C \otimes_{\Q} E(-4)_{dR}\\
\end{CD}
\end{equation}
where the vertical lines are the comparison isomorphisms.

\begin{lem} \label{nullite} The pairing
$$
\begin{CD}
M_B(\Pi_f \times \sigma_f, 2)_{\R}^+ @>\langle \,\,, \overline{\omega}_\Psi \otimes \eta_\Phi \rangle_{B, \C}>> \C \otimes_{\Q} E(-2)_B
\end{CD}
$$
with $\overline{\omega} \otimes \eta$, induces an $\R \otimes_{\Q} E$-linear map
$$
\begin{CD}
\mrm{Ext}^1_{\mrm{MHS}_{\R}^+}(\R(0), M_B(\Pi_f \times \sigma_f, 3)_{\R}) @>\langle \,\,, \overline{\omega}_\Psi \otimes \eta_\Phi \rangle_{B, \C}>> \C \otimes_{\Q} E(-2)_B.
\end{CD}
$$
\end{lem}

\begin{proof} In the Hodge decomposition (\ref{hodge-dec}), we have $\overline{\omega}_\Psi \otimes \eta_\Phi \in M^{2, 2}$. Moreover the image of the natural map $F^3 M_{dR}(\Pi_f \times \sigma_f)_{\R} \rightarrow M_B(\Pi_f \times \sigma_f, 2)_{\R}^+$ (see Prop. \ref{shortexact}) lies in $M^{4, 0} \oplus M^{3, 1} \oplus M^{1, 3} \oplus M^{0, 4}$. As the Poincar\'e duality pairing is a morphism  of Hodge structures, the restriction of
$$
\begin{CD}
M_B(\Pi_f \times \sigma_f, 2)_{\R}^+ @>\langle \,\,, \overline{\omega}_\Psi \otimes \eta_\Phi \rangle_{B, \C}>> \C \otimes_{\Q} E(-2)_B
\end{CD}
$$
to $F^3 M_{dR}(\Pi_f \times \sigma_f)_{\R}$ is zero. Hence we get the induced map 
$$
\begin{CD}
\mrm{Ext}^1_{\mrm{MHS}_{\R}^+}(\R(0), M_B(\Pi_f \times \sigma_f, 3)_{\R}) @>\langle \,\,, \overline{\omega}_\Psi \otimes \eta_\Phi \rangle_{B, \C}>> \C \otimes_{\Q} E(-2)_B.
\end{CD}
$$
\end{proof}

Recall that we have fixed vectors $\Psi_f$ and $\Phi_f$ in the representation space of $\Pi_f$ and $\sigma_f$ respectively. Hence, we can consider the image $v_\mathcal{Z}$ of the cycle class $\mathcal{Z}_N$ (\ref{ZN}) by the maps
$$
H^4_B(S(N) \times Y(N), \Q(2))^+ \rightarrow H^4_B(S \times Y, \Q(2))^+ \otimes E \rightarrow M_B(\Pi_f \times \sigma_f, 2)^+ 
$$
where the second map sends a vector $x \in H^4_B(S \times Y, \Q(2))^+ \otimes E$ to the unique element of $M_B(\Pi_f \times \sigma_f, 2)^+$ sending $\Psi_f \times \sigma_f$ to $x$. It is easy to see that the vector $v_\mathcal{Z}$ is mapped to the $E$-subspace
$
\mathcal{Z}(\Pi_f \times \sigma_f)$
by the right hand map 
$$
M_B(\Pi_f \times \sigma_f, 2)^+ \rightarrow \mrm{Ext}^1_{\mrm{MHS}_{\R}^+}(\R(0), M_B(\Pi_f \times \sigma_f, 3)_{\R})
$$
of the short exact sequence of Prop. \ref{shortexact}.\\

We now define a Haar measure on $H(\A)/\R^\times_+$ as follows. For every prime number $p$, we endow $H(\Q_p)$ with the unique Haar measure $dh_p$ for which $H(\Z_p)$ has volume one. Let us consider 
$$
U_\infty=(\R^\times_+\mrm{SO}(2, \R) \times \R^\times_+\mrm{SO}(2, \R))\cap H(\R).
$$
Then $U_\infty$ is a subgroup of $H(\R)_+=\{(h_1, h_2) \in H(\R)\,|\, \det h_1 = \det h_2 >0\}$ which is maximal compact modulo the center. Let $\mathfrak{h}$ and $\mathfrak{u}$ be the complex Lie algebras of $H(\R)$ and $U_\infty$ respectively. Then
$$
\textbf{1}=(v^+, 0) \wedge (0, v^+) \wedge (v^-, 0) \wedge (0, v^-)
$$ 
is a generator of the highest exterior power $\bigwedge^4 \mathfrak{h}/\mathfrak{u}$ and it determines a left invariant measure $dx$ on $H(\R)/U_\infty$. Together with the Haar measure $dk$ on $\mrm{SO}(2, \R) \times \mrm{SO}(2, \R)$ whose total mass is one, this defines a measure $dh_\infty=dx dk$ on $H(\R)/\R^\times_+$. Let $dh$ denote the product measure $\prod_{v \leq \infty} dh_v$ on $H(\A)/\R^\times_+$. 

\begin{defn} For two elements $x, y \in \C \otimes_{\Q} E$, we write $x \sim y$ if $x$ and $y$ belong to the same orbit under right multiplication by $E^\times$.
\end{defn}

Let $Z$ denote the center of $H$. In the following proposition, we regard a complex number as an element of $\C \otimes_{\Q} E=\prod_{\sigma: E \rightarrow \C} \C$ via the diagonal inclusion $\C \rightarrow \prod_{\sigma: E \rightarrow \C} \C$. 

\begin{pro} \label{coho-period} $$
\langle v_\mathcal{Z}, \overline{\omega}_\Psi \otimes \eta_\Phi \rangle_{B, \C} \sim  \int_{Z(\A) H(\Q) \backslash H(\A)} (X_{(-1, 1)} \overline{\Psi})(h_1, h_2) \Phi(h_2) dh
$$
\end{pro}

\begin{proof} Recall that we denote by $\mathcal{Z}_N \in H^4_B(S(N) \times Y(N),\Q(2))^+$ the cohomology class of $\iota_N(Y(N) \times_{\Q(\mu_N)} Y(N))$. According to \cite{griffiths-harris} Ch. 3, example 1 p. 386, the image  of $\mathcal{Z}_N$ in the de Rham cohomology space $H^4_{dR}(S(N) \times Y(N), \C(2))$ is given by the integration current along the subvariety $Y(N) \times_{\Q(\mu_N)} Y(N)$. This means that for any closed differential form with compact support $\Omega$ of degree $4$ on $S(N) \times Y(N)$, the Poincar\'e duality pairing $
\langle \mathcal{Z}_N, \Omega \rangle$ is given by $$
\langle \mathcal{Z}_N, \Omega \rangle=\int_{Y(N) \times_{\Q(\mu_N)} Y(N)} \Omega.
$$
By the invariance properties of $\Psi_f$ and $\Phi_f$, the element $\overline{\Omega}_\Psi \otimes \Omega_{\Phi}$ defines  a harmonic differential form on $S(N) \times Y(N)$. It is not compactly supported, however it is cuspidal. As a consequence, by \cite{borel0} Cor. 5.5, there exists a rapidly decreasing differential form $\Omega'$ such that the difference $\overline{\Omega}_\Psi \otimes \Omega_\Phi-d\Omega'$ is compactly supported. The integral
$$
\int_{Y(N) \times_{\Q(\mu_N)} Y(N)} d \Omega'
$$
converges and is equal to zero. This follows from the fact that, as $\Omega'$ is rapidly decreasing, it extends to a smooth differential form on a smooth compactification of $Y(N) \times_{\Q(\mu_N)} Y(N)$ which is zero on the boundary and from the Stokes theorem. As a consequence
\begin{eqnarray*}
\langle v_\mathcal{Z}, \overline{\omega}_\Psi \otimes \eta_\Phi \rangle_{dR, \C} &=& \int_{Y(N) \times_{\Q(\mu_N)} Y(N)} (\overline{\Omega}_\Psi \otimes \Omega_\Phi-d \Omega')\\
&=& \int_{Y(N) \times_{\Q(\mu_N)} Y(N)} \overline{\Omega}_\Psi \otimes \Omega_\Phi.
\end{eqnarray*}
Let $U(N)$ denote the principal level $N$ congruence subgroup of $H(\widehat{\Z})$. As complex analytic varieties, we have
$$
Y(N) \times_{\Q(\mu_N)} Y(N) \simeq H(\Q) \backslash H(\A)/U_\infty U(N).
$$
Let $h_N$ denote the cardinality of the finite group $Z(\Q)\backslash Z(\A_f)/U(N) \cap Z(\A_f)$. We have
$$
\int_{Y(N) \times_{\Q(\mu_N)} Y(N)} \overline{\Omega}_\Psi \otimes \Omega_\Phi =  \int_{H(\Q) \backslash H(\A)/U_\infty U(N)} (\overline{\Omega}_\Psi \otimes \Omega_\Phi)(\textbf{1})dh
$$
\begin{eqnarray*}
&=&  \int_{H(\Q) \backslash H(\A)/U_\infty U(N)} \overline{\Omega}_\Psi((v^-, 0) \wedge (0, v^-) \otimes (v^+, 0)) \Omega_\Phi(v^+) dh\\
&=&  \int_{\R^\times_+ H(\Q) \backslash H(\A)/U(N)} \overline{\Omega}_\Psi((v^-, 0) \wedge (0, v^-) \otimes (v^+, 0)) \Omega_\Phi(v^+) dh\\
&=& 4 h_N \int_{Z(\A)H(\Q) \backslash H(\A)/U(N)}\overline{\Omega}_\Psi((v^-, 0) \wedge (0, v^-) \otimes (v^+, 0)) \Omega_\Phi(v^+) dh\\
&=& \frac{4 h_N}{\mrm{vol}(U(N))} \int_{Z(\A)H(\Q) \backslash H(\A)}\overline{\Omega}_\Psi((v^-, 0) \wedge (0, v^-) \otimes (v^+, 0)) \Omega_\Phi(v^+) dh
\end{eqnarray*}
By an explicit computation similar to the ones conducted in the proof of \cite{lemma2} Lem. 4.27, one proves that 
$$
(v^-, 0) \wedge (0, v^-) \otimes (v^+, 0)=r\ad_{X_{(-1,1)}}(X_{(0,-2)} \wedge X_{(-1,-1)} \otimes X_{(2, 0)})
$$
for some explicit $r \in \Q^\times$ that we do not need to compute as we work up to $\sim$. Note that $\mrm{vol}(U(N))$ is a non-zero rational number with our choice of measure. So the conclusion follows.
\end{proof}

\subsubsection{The pairing associated to the Deligne rational structure} \label{pairing-deligne}

\begin{pro} \label{poles} If the complete $L$-function $L(s, \Pi \times \sigma)$ has a pole at $s=1$ then $\Pi$ is obtained as a Weil lifting from $\mrm{GSO}(2, 2, \A)$ of a pair $(\sigma_1, \sigma_2)$ of inequivalent irreducible cuspidal automorphic representations of $\mrm{GL}(2, \A)$ with the same central characters. 
\end{pro} 

\begin{proof} This follows from \cite{pssoudry} Thm. 1.3 and \cite{weissauer2} Lem. 5.2 1) as we assume that $\Pi$ is globally generic hence is not CAP.
\end{proof}

\begin{pro} \label{mult} Assume that $L(s, \Pi \times \sigma)$ has a pole at $s=1$. Then 
$$
m(\Pi_\infty^H \otimes \Pi_f)=0.
$$
\end{pro}

\begin{proof} This follows from the previous result and from \cite{weissauer2} Thm. 5.2. (4) and Lem. 5.2.
\end{proof}

In the rest of section \ref{pairing-deligne}, we assume that $L(s, \Pi \times \sigma)$ has a pole at $s=1$ so that the conclusion of Prop. \ref{mult} holds. Hence, by the Hodge decompositions (\ref{hodge}) and (\ref{hodge2}), we have the decompositions
\begin{eqnarray*}
M_B(\Pi_f) &=& M_B(\Pi_f)^+ \oplus M_B(\Pi_f)^-,\\
M_B(\sigma_f) &=& M_B(\sigma_f)^+ \oplus M_B(\sigma_f)^-
\end{eqnarray*}
where $M_B(\Pi_f)^\pm$ and $M_B(\sigma_f)^\pm$ are one-dimensional $E$-vector spaces. Let $v^\pm$ and $w^\pm$ be generators of $M_B(\Pi_f)^\pm$ and $M_B(\sigma_f)^\pm$ respectively. By a slight abuse of notation, we regard $(v^+, v^-)$ as a basis of $M_B(\Pi_f)_{\C}$ and $(w^+, w^-)$ as a basis of $M_B(\sigma_f)_{\C}$. Let $\omega$ be a generator of the one dimensional $E$-vector space $F^2 M_{dR}(\Pi_f)$ and let $\eta$ be a generator of the one dimensional $E$-vector space $F^1 M_{dR}(\sigma_f)$. Then $\omega \otimes \eta$ is a generator of $F^3 M_{dR}(\Pi_f \times \sigma_f)$. Via the comparison isomorphisms, we have
\begin{eqnarray*}
\omega &=& \alpha^+ v^+ + \alpha^- v^-,\\
\eta &=& \beta^+ w^+ + \beta^-  w^-.
\end{eqnarray*}
for some $\alpha^+, \alpha^-, \beta^+, \beta^- \in \C \otimes E$. Note that as $\omega$ and $\eta$ are defined over $E$ in the de Rham rational structure, we have $\alpha^+, \beta^+ \in \R \otimes_{\Q} E$ and $\alpha^-, \beta^- \in \R i\otimes E$. The image of $\omega \otimes \eta$ by the natural map $\phi: F^3 M_{dR}(\Pi_f \times \sigma_f)_{\R} \rightarrow M_B(\Pi_f \times \sigma_f, 2)^+_{\R}$ in the exact sequence of Prop. \ref{shortexact} is
$$
\phi(\omega \otimes \eta)=(2 \pi i)^2 (\alpha^+ \beta^+ v^+ \otimes w^+ + \alpha^- \beta^- v^- \otimes w^-).
$$
As $\phi$ is injective, at least one of the two real numbers $\alpha^+ \beta^+$ and $\alpha^- \beta^-$ is non-zero. 

\begin{lem} \label{q-str-beilinson} Let $\circ \in \{\pm\}$ be such that $\alpha^\circ \beta^\circ \neq 0$ then 
$$
v_\mathcal{B}=\frac{1}{(2 \pi i)^2 \alpha^\circ \beta^\circ} v^{-\circ}\otimes w^{-\circ}
$$ maps to a generator of $\mathcal{B}(\Pi_f \times \sigma_f)$ by the right hand map 
$$
M_B(\Pi_f \times \sigma_f, 2)^+ \rightarrow \mrm{Ext}^1_{\mrm{MHS}_{\R}^+}(\R(0), M_B(\Pi_f \times \sigma_f, 3)_{\R})
$$
of the short exact sequence of Prop. \ref{shortexact}. 
\end{lem}

\begin{proof} For any $\lambda^+, \lambda^- \in \R \otimes_{\Q} E$, the vector $\lambda^+ v^+ \otimes w^+ + \lambda^- v^- \otimes w^- \in M_B(\Pi_f \times \sigma_f, 2)^+$ maps to a generator $\mathcal{B}(\Pi_f \times \sigma_f)$ if and only if $(2 \pi i)^2(\lambda^+ \alpha^- \beta^- - \lambda^- \alpha^+ \beta^+) \in E^\times$. This can be proved in exactly the same way as \cite{lemma2} Lem. 6.1. The conclusion follows.
\end{proof}

Let $\overline{\omega}$ be the vector in $M_B(\Pi_f \times \sigma_f)_{\C}$ obtained by applying the complex conjugation on the coefficients. Then $\overline{\omega}=\alpha^+  v^+  - \alpha^- v^- $. Let us consider
$$
 \overline{\omega} \otimes \eta = \alpha^+ \beta^+ v^+ \otimes w^+ + \alpha^+ \beta^- v^+ \otimes w^- - \alpha^- \beta^+ v^- \otimes w^+ - \alpha^- \beta^- v^- \otimes w^-.
$$

\begin{lem} \label{pairing1} We have $\langle v_\mathcal{B}, \overline{\omega} \otimes \eta \rangle_{B, \C} \sim (2\pi i)^{-6}$.
\end{lem}

\begin{proof} Poincar\'e duality for $M_B(\Pi_f)$ is a perfect $F_\infty$-equivariant pairing $$M_B(\Pi_f) \otimes_E M_B(\Pi_f, 3) \rightarrow E.$$ As a consequence $v^\pm$ is dual to $(2\pi i)^3 v^\mp$ for this pairing, up to multiplication by an element of $E^\times$. Similarly, the vector $w^\pm$ is dual to $(2 \pi i) w^\mp$ for the pairing $$M_B(\sigma_f) \otimes_E M_B(\sigma_f, 1) \rightarrow E,$$ up to multiplication by an element of $E^\times$. Then the conclusion follows.
\end{proof}

\begin{pro} \label{delta} $\delta(\Pi_f \times \sigma_f, 3) \sim (2 \pi i)^{4}$
\end{pro}

\begin{proof} Let $e_1, e_2, e_3, e_4$ be the vectors of a basis of $M_B(\Pi_f \times \sigma_f)$, let $f_1, f_2, f_3, f_4$ be the vectors of a basis of $M_{dR}(\Pi_f \times \sigma_f)$, let $P$ be the matrix of the vectors $I_\infty(e_i)$ in the basis $(f_j)$ and let us denote $\delta(\Pi_f \times \sigma_f)=\det P$. Let us consider the matrices $J_B=(\langle e_i, e_j \rangle_B)$ and $J_{dR}=(\langle f_i, f_j \rangle_{dR})$. Let $i_\infty$ be the comparison isomorphism which is the right vertical line of the diagram (\ref{vertical}). Then we have $P^t J_{dR} P=i_\infty(J_B)$ where $i_{\infty}(J_B)$ denotes the matrix obtained by applying $i_\infty$ to the coefficients of $J_B$.
Note that the matrices $J_{dR}$ and $(2\pi i)^4 i_\infty(J_B)$ have coefficients in $E$. The last statement follows from the computation of the periods of the Tate motive as explained in \cite{valeurs-deligne} 3. Taking determinants, we see that $\delta(\Pi_f \times \sigma_f)^2 \sim (2 \pi i)^{-16}$ so that we have $\delta(\Pi_f \times \sigma_f) \sim (2 \pi i)^{-8}$. As $M(\Pi_f \times \sigma_f)$ has rank four, it follows from \cite{valeurs-deligne} (5.1.9) that 
$
\delta(\Pi_f \times \sigma_f, 3) \sim (2 \pi i)^{12} \delta(\Pi_f \times \sigma_f) 
$ and the conclusion follows.
\end{proof}

\begin{cor} \label{un-corollaire} Let $\omega$ be a generator of $F^2 M_{dR}(\Pi_f)$ and let $\eta$ be a generator of $F^1 M_{dR}(\sigma_f)$. Then
$$
\mathcal{Z}(\Pi_f \times \sigma_f) = (2\pi i)^8 \langle v_\mathcal{Z}, \overline{\omega} \otimes \eta \rangle_{B, \C} \mathcal{D}(\Pi_f \times \sigma_f).
$$
\end{cor}

\begin{proof} Recall that, by definition
$$
\mathcal{D}(\Pi_f \times \sigma_f) = (2\pi i)^{\dim_{E}M_B(\Pi_f \times \sigma_f, 3)^-} \delta(\Pi_f \times \sigma_f, 3)^{-1} \mathcal{B}(\Pi_f \times \sigma_f) = (2 \pi i)^{-2} \mathcal{B}(\Pi_f \times \sigma_f)
$$
where the last equality follows from Prop. \ref{delta}. Let $\overline{v}_\mathcal{Z}$ and $\overline{v}_\mathcal{B}$ be the images of $v_\mathcal{Z}$ and $v_\mathcal{B}$ by the third map of the short exact sequence of Prop. \ref{shortexact}. There exists $\lambda \in \R \otimes_{\Q} E$ such that $\overline{v}_\mathcal{Z}=\lambda \overline{v}_\mathcal{B}$ and then $$\mathcal{Z}(\Pi_f \times \sigma_f) = \lambda \mathcal{B}(\Pi_f \times \sigma_f) = (2\pi i)^2 \lambda \mathcal{D}(\Pi_f \times \sigma_f).$$ Pairing with $\overline{\omega} \otimes \eta$ and using Lem. \ref{pairing1}, we obtain 
$$
\langle v_\mathcal{Z} , \overline{\omega} \otimes \eta \rangle=\lambda \langle v_\mathcal{B}, \overline{\omega} \otimes \eta \rangle \sim (2 \pi i)^{-6} \lambda.
$$
\end{proof}

\section{Zeta integrals} \label{zetaint}

\subsection{The global integral} \label{section-zeta}  Let $\psi: \Q \backslash \A \rightarrow \C^\times$ be the non-trivial additive character characterized by $\psi(x)=e^{2\pi i x}$ for $x \in \R$. We consider the maximal unipotent subgroup $N \subset \mrm{GSp}(4)$ defined by
\begin{equation} \label{unipotent}
N=\left\{ n(x_0, x_1, x_2, x_3)=\begin{pmatrix}
1 & x_0 &  & \\
 & 1 &  & \\
 &  & 1 & \\
 &  & -x_0 & 1\\
\end{pmatrix}\begin{pmatrix}
1 & & x_1 & x_2 \\
 & 1 & x_2 & x_3\\
 &  & 1 & \\
 &  &  & 1\\
\end{pmatrix}, x_0, x_1, x_2, x_3 \in \mathbb{G}_a \right\}
\end{equation}
and the character $\psi_N: N(\Q) \backslash N(\A) \rightarrow \C$ defined by
$
\psi_N(n(x_0, x_1, x_2, x_3))=\psi(-x_0-x_3).
$\\

Let $\Pi=\bigotimes'_v \Pi_v$ be an irreducible cuspidal automorphic representation of $\mrm{GSp}(4, \A)$. The global Whittaker function $W_\Psi$ on $\mrm{GSp}(4, \A)$ attached to a cusp form $\Psi \in \Pi$ is
\begin{equation} \label{global-whittaker}
W_\Psi(g)=\int_{N(\Q)\backslash N(\A)} \Psi(ng) \psi_N(n^{-1})dn.
\end{equation}
Assume that the global Whittaker function $W_\Psi$ does not vanish for some cusp form $\Psi \in \Pi$. This assumption implies that for each place $v$, the representation $\Pi_v$ of $\mrm{GSp}(4 \Q_v)$ can be realized as a subspace of
$$
\{W: \mrm{GSp}(4, \Q_v) \rightarrow \C \,|\, \text{smooth}\,,\,W(ng)=\psi_N(n)W(g),\,\forall(n, g) \in N(\Q_v) \times \mrm{GSp}(4, \Q_v)\}.
$$
We denote this subspace by $W(\Pi_v, \psi_v)$ and call it the local Whittaker model of $\Pi_v$. If $\sigma$ is a cuspidal automorphic representation of $\mrm{GL}(2, \A)$ and if $\Phi \in \sigma$ is a cusp form, the global Whittaker function on $\mrm{GL}(2, \A)$ attached to $\Phi$ is
$$
W_\Phi(g)=\int_{\Q \backslash \A} \Phi \left( \begin{pmatrix}
1 & x\\
  & 1\\
\end{pmatrix} g\right) \psi(-x)dx.
$$
It is well-known that $W_\Phi$ does not vanish for $\Phi \neq 0$. The local Whittaker model $W(\sigma_v, \psi_v)$ of $\sigma_v$ is defined similarly as above.\\

Let us now introduce Eisenstein series on $\mrm{GL}(2, \A)$. Let $d t_\infty$ be the Lebesgue measure on the additive group $\mbb{R}$. If $v$ is a non-archimedean place of $\mbb{Q}$, let $d t_v$ be the Haar measure on $\Q_v$ for which $\mbb{Z}_v$ has volume one. Let $d^\times t_v$ be the Haar measure on $\mbb{Q}_v^\times$ defined by
$$
d^\times t_v = \left\{
\begin{array}{ll}
        \frac{dt_v}{|t_v|} & \mbox{if } v \mbox{ is archimedean},\\
        \frac{p}{p-1} \frac{dt_v}{|t_v|} & \mbox{if } v \mbox{ is } p\mbox{-adic}.
\end{array}
\right.
$$
Let $d t$, resp. $d^\times t$, denote the product measure $\prod_v d t_v$ on $\mbb{A}$, resp. $\prod_v d^\times t_v$ on $\mbb{A}^\times$. Let $\mathcal{S}(\A^2)$ be the space of Schwartz-Bruhat functions on $\A^2$. For $\varphi \in \mathcal{S}(\A^2)$, let us define the global Jacquet section
$$
f_\varphi(s, h_1)=|\det h_1|^s\int_{\A^\times} \varphi((0,t)h_1)|t|^{2s} d^\times t
$$
for $h_1 \in \mrm{GL}(2, \A)$. This integral converges for $\re(s)>1/2$. We can define the Eisenstein series
$$
E(s, h_1, f_\varphi)=\sum_{\gamma \in B(\Q) \backslash \mrm{GL}(2, \Q)} f_\varphi(s, \gamma h_1)
$$
which converges absolutely and uniformly on every compact subset in $\re(s)>1$ except for the poles of $f_\varphi(s,h_1)$ and is continued to a meromorphic function on the whole complex plane. The global zeta integral we are interested in is defined as follows: for $\Psi \in \Pi$, $\Phi \in \sigma$ and $\varphi \in \mathcal{S}(\A^2)$ let 
$$
\mathcal{Z}(s, \Psi, \Phi, f_\varphi)=\int_{Z(\A)H(\Q) \backslash H(\A)} \Psi(h)\Phi(h_2)E(s, h_1, f_\varphi) dh.
$$
Here and in what follows, we regard the group $H$ as embedded in $\mrm{GSp}(4)$ via $p\circ \iota$, where $p: G \rightarrow \mrm{GSp}(4)$ is the first projection. This integral converges absolutely except for the poles of the Eisenstein series and defines a meromorphic function in $s \in \C$.

\begin{pro} \label{residue} The function $s \mapsto \mathcal{Z}(s, \Psi, \Phi, f_\varphi)$ is holomorphic except for possible simple poles at $s=1$ and $0$. Moreover, we have
$$
\Res_{s=1} \mathcal{Z}(s, \Psi, \Phi, f_\varphi)=  \frac{\widehat{\varphi}(0)}{2} \int_{Z(\A) H(\Q) \backslash H(\A)} \Psi(h) \Phi(h_2) dh
$$
where $$\widehat{\varphi}(0)=\int_{\A^2} \varphi(s,t)dsdt.$$
\end{pro}

\begin{proof} This is a particular case of \cite{moriyama1} Prop. 3.1. Note that the constant $c$ appearing in the statement of loc. cit., Prop. 3.1 is equal to $\frac{1}{2}$ when the ground field is $\Q$.
\end{proof}

Let us suppose that the cusp forms $\Psi=\bigotimes'_v \Psi_v$ and $\Phi=\bigotimes_v' \Phi_v$ are factorizable. Then, the local multiplicity one property implies that $W_\Psi$ and $W_\Phi$ are decomposed into a product of local Whittaker functions:
\begin{eqnarray*}
W_\Psi(g) &=& \prod_v W_{\Psi_v}(g_v), g \in \mrm{GSp}(4, \A)\\
W_\Phi(h_2) &=& \prod_v W_{\Phi_v}(h_{2, v}), h_2 \in \mrm{GL}(2, \A).
\end{eqnarray*}
Moreover, let us assume that the Schwartz-Bruhat function $\varphi=\prod_v \varphi_v$ is factorizable. Then the global Jacquet section $f_\varphi$ factorizes as $f_\varphi(s,h_1)=\prod_v f_{\varphi_v}(s, h_{1,v})$ where
$$
f_{\varphi_v}(s, h_{1,v})= |\det h_{1,v}|_v^s \int_{\Q_v^\times} \varphi_v((0, t_v) h_{1,v}) |t_v|_v^{2s} d^\times t_v.
$$
For each place $v$ of $\Q$, we define the local zeta integral $\mathcal{Z}_v(s, W_{\Psi_v}, W_{\Phi_v}, f_{\varphi_v})$ by
$$
\mathcal{Z}_v(s, W_{\Psi_v}, W_{\Phi_v}, f_{\varphi_v})=\int_{Z(\Q_v) N_H(\Q_v) \backslash H(\Q_v)}W_{\Psi_v}(h_v)W_{\Phi_v}(h_{2,v}) f_{\varphi_v}(s, h_{1,v})dh_v
$$
where $N_H$ denotes the maximal unipotent subgroup of $H$ defined as $N_H=N \cap H$. 

\begin{pro}\cite{moriyama1} Prop. 3.2. Suppose that $\mathcal{Z}_\infty(s, W_{\Psi_\infty}, W_{\Phi_\infty}, f_{\varphi_\infty})$ converges absolutely for $\re(s)>e_\infty$. Then, the integral
$$
\int_{Z(\A) N_H(\A) \backslash H(\A)} W_\Psi(h) W_\Phi(h_2) f_\varphi(s, h_1) dh
$$
converges absolutely for $\re(s)>\max\{3, e_\infty\}$ and is equal to $\mathcal{Z}(s, \Psi, \Phi, f_\varphi)$.
\end{pro}

This proposition implies that
$$
\mathcal{Z}(s, \Psi, \Phi, f_\varphi)=\prod_v \mathcal{Z}_v(s, W_{\Psi_v}, W_{\Phi_v}, f_{\varphi_v})
$$
for any complex number $s$ such that $\re(s)> \max\{3, e_\infty\}$. Let $V$ be a finite set of places satisfying the following condition: if $v \notin V$, then $v$ is a finite place such that $\psi_v$ is unramified, the representations $\Pi_v$, resp. $\sigma_v$, has a vector fixed by $\mrm{GSp}(4, \Z_v)$, resp. $\mrm{GL}(2, \Z_v)$ and $\varphi_v \in \mathcal{S}(\Q_v^2)$ is the characteristic function $\varphi_{v,0}$ of $\Z_v^2$. Fix a place $v \notin V$. Let $W_0 \in W(\Pi_v, \psi_v)$ and $W'_0 \in W(\sigma_v, \psi_v)$ be the $\mrm{GSp}(4, \Z_v)$-fixed local Whittaker function normalized so that $W_0(I_4)=1$ and the $\mrm{GL}(2, \Z_v)$-fixed Whittaker function normalized so that $W'_0(I_2)=1$, respectively. Then, it is proved in section 3.4 of \cite{moriyama1} that
$$
 \mathcal{Z}_v(s, W_0, W'_0, f_{\varphi_{v,0}})=L(s, \Pi_v \times \sigma_v)
$$
where the right-hand side is the Langlands degree eight local $L$-factor. As a consequence, for $\re(s)>\max\{3, e_\infty\}$, we have
$$
\mathcal{Z}(s, \Psi, \Phi, f_\varphi)=\prod_{v \in V}\mathcal{Z}_v(s, W_{\Psi_v}, W_{\Phi_v}, f_{\varphi_v}) L_V(s, \Pi \times \sigma).
$$
In the follwing result, we denote by $\mathcal{S}(\Q_v^2)$ the space of $\C$-valued locally constant compactly supported functions on $\Q_v^2$.

\begin{pro} \label{ram-int-soudry} \cite{soudry} section 2. Let $v$ be a non-archimedean place. There exists a unique polynomial $P_v(X) \in \C[X]$ such that $P_v(0)=1$ and that the $\C$-vector space generated by the $\mathcal{Z}_v(s, W_{\Psi_v}, W_{\Phi_v}, f_{\varphi_v})$ for $W_{\Psi_v} \in W(\Pi_v, \psi_v)$, $W_{\Phi_v} \in W(\sigma_v, \psi_v)$ and $ \varphi_v \in \mathcal{S}(\Q_v^2)$ is $P_v(p^{-s}) \C[p^{-s}, p^s]$.
\end{pro}

For any non-archimedean place $v \in V$ let $L_v(s, \Pi \times \sigma)=P_v(p^{-s})^{-1}$.

\subsection{Archimedean computation} The result of this section is a particular case of \cite{moriyama1}, that we wish to report for convenience of the reader. First we need to fix the notation for our archimedean representations.\\

Recall that we denote by $\sigma_\infty^{1,0}$, resp. $\sigma_\infty^{0,1}$, the holomorphic, resp. antiholomorphic, discrete series of $\mrm{GL}(2, \R)_+$ with the same central and infinitesimal characters as the trivial representation. With the notation of section 1.1 (i) of loc. cit., we have
\begin{eqnarray*}
\sigma_\infty^{1,0}|_{\mrm{SL}(2,\R)} &=& D_{2},\\
\sigma_\infty^{0,1}|_{\mrm{SL}(2,\R)} &=& D_{-2}.
\end{eqnarray*}
Recall also that we consider the discrete series $\Pi_\infty^{2,1}$ of $\mrm{GSp}(4, \R)_+$ with trivial central character which contains with multiplicity one the irreducible $\C[K_\infty]$-module $\tau_{(3,-1)}$ as a minimal $K_\infty$-type. In other words $\Pi_\infty^{2,1}$ has Blattner parameter $(3, -1)$ and the discrete series $\Pi_\infty^{1,2}$ has Blattner parameter $(1,-3)$. In the notation of the section 1.2 (i) of loc. cit. we have
\begin{eqnarray*}
\Pi_\infty^{2,1}|_{\mrm{Sp}(4,\R)} &=& D_{(3,-1)},\\
\Pi_\infty^{1,2}|_{\mrm{Sp}(4,\R)} &=& D_{(1,-3)}.
\end{eqnarray*}
Let $\Psi_\infty$ denote a highest weight vector of the minimal $K_\infty$-type of $\Pi_\infty^{2,1}$. Then $\overline{\Psi}_\infty$, which is defined as in \cite{lemma2} Rem. 3.2, is a highest weight vector of the minimal $K_\infty$-type of $\Pi_\infty^{1,2}$. Let $\Phi_\infty$ denote a generator of the minimal $L_\infty$-type of $\sigma_\infty$. Let us fix an isomorphism $\Pi_\infty^W \simeq W(\Pi_\infty^W, \psi_\infty)$ and let $W_{X_{(-1,1)}\overline{\Psi}_\infty}$ be the image of $X_{(-1,1)}\overline{\Psi}_\infty$ under this isomorphism. Similarly let $W_{\Phi_\infty}$ denote the image of $\Phi_\infty$ under a fixed isomorphism $\sigma_\infty \simeq W(\sigma_\infty, \psi_\infty)$.

\begin{pro} \label{archi-comp} Let $\Gamma_{\C}(s)$ denote the function $\Gamma_{\C}(s)=2(2\pi)^{-s}\Gamma(s)$. Let $\varphi_\infty \in \mathcal{S}(\R^2)$ be the Schwartz-Bruhat function defined as 
$
\varphi_\infty(x,y)=\exp(-\pi(x^2+y^2)).
$ Normalize the functions $W_{X_{(-1,1)}\overline{\Psi}_\infty}$ and $W_{\Phi_\infty}$ by $W_{X_{(-1,1)}\overline{\Psi}_\infty}(I_4)=1$ (see \cite{moriyama1} (5.4)) and $W_{\Phi_\infty}(I_2)=c_2$, where $I_4$ and $I_2$ denote the identity matrices of size $4$ and $2$ and where $c_2$ is the unique non-zero complex number such that the constant $C$ appearing in \cite{moriyama1} (5.11) is equal to $1$. Then, for any $\re(s)>e_\infty$, we have
$$
\mathcal{Z}_\infty(s, W_{X_{(-1,1)}\overline{\Psi}_\infty}, W_{\Phi_\infty}, f_{\varphi_\infty})= \Gamma_{\C}\left(s+2\right) \Gamma_{\C}\left(s+1\right)^2 \Gamma_{\C}\left(s\right).
$$
\end{pro}

\begin{proof} This is a particular case of the equality (5.11) in \cite{moriyama1} with $\lambda_1=3$, $\lambda_2=-1$, $l=2$.
\end{proof}

\begin{rems} It would be possible to determine $c_2$ explicitly by unravelling carefully the computations of \cite{moriyama1} 5.3. 
\end{rems}

\subsection{De Rham-Whittaker periods and rationality of $p$-adic integrals} \label{dR-W} In this section, we define de Rham-Whittaker periods and prove an algebraicity result for $p$-adic integrals. De Rham-Whittaker periods are analogous to the occult period invariant introduced in \cite{occult}, when the Bessel model is replaced by the Whittaker model. Note that in this section, we do not assume that $L(s, \Pi \times \sigma)$ has a pole at $s=1$.\\

\subsubsection{Definition of the periods} We start be defining an action of $\mrm{Aut}(\C)$ on the Whittaker model. Let $\Q(\mu_\infty)$ denote the extension of $\Q$ generated by roots of unity of arbitrary order. By choosing a compatible system of roots of unity, we obtain a natural morphism $\mrm{Gal}(\Q(\mu_\infty)/\Q) \rightarrow \widehat{\Z}^\times$ which in fact is an isomorphism. Let us denote by $\sigma \mapsto t_\sigma = \prod_p t_{\sigma, p}$ the following composite 
$$
\mrm{Aut}(\C) \rightarrow \mrm{Gal}(\overline{\Q}/\Q) \rightarrow \mrm{Gal}(\Q(\mu_\infty)/\Q) \rightarrow \widehat{\Z}^\times \simeq \prod_p \Z_p^\times
$$
Let $W \in W(\Pi_p, \psi_p)$ be a Whittaker function and let $\sigma \in \mrm{Aut}(\C)$. We define $^\sigma W$ by $^\sigma W(g)=\sigma(W(T_{\sigma, p} g))$ where $T_{\sigma, p} \in \mrm{GSp}(4, \Z_p)$ is defined by
$$
T_{\sigma, p}=\begin{pmatrix}
t_{\sigma, p}^{-3} & &  & \\
 &  t_{\sigma, p}^{-2} &  & \\
 &  & 1 & \\
 & &  & t_{\sigma, p}^{-1}\\
\end{pmatrix}
$$
Note that for any $x_0, x_1, x_2, x_3 \in \Q_p$, we have 
$$
T_{\sigma, p}n(x_0, x_1, x_2, x_3)T_{\sigma, p}^{-1}=n(t_{\sigma, p}^{-1}x_0, t_{\sigma, p}^{-3}x_1, t_{\sigma, p}^{-2}x_2, t_{\sigma, p}^{-1}x_3)
$$
where $n(x_0, x_1, x_2, x_3) \in N(\Q_p)$ is the element defined in (\ref{unipotent}). Using the fact that $\sigma(\psi_p(t_{\sigma, p}^{-1}x))=\psi_p(x)$ for any $x \in \Q_p$, this shows that $^\sigma W \in W(\Pi_p, \psi_p)$. The map $W \mapsto\,\!^\sigma W$ defines a $\sigma$-linear intertwining operator $W(\Pi_p, \psi_p) \rightarrow W(\,\!^\sigma \Pi_p, \psi_p)$. Let $W(\Pi, \psi)=W(\Pi_\infty, \psi_\infty) \otimes W(\Pi_f, \psi_f)$ be the Whittaker model of $\Pi$. Similarly as above, we define a $\sigma$-linear intertwining operator $\tilde{\sigma}: W(\Pi_f, \psi_f) \rightarrow W(\,\!^\sigma \Pi_f, \psi_f)$. Note that if $\Pi_p$ is unramified, then $W \mapsto \,\!^\sigma W$ sends a spherical vector to a spherical vector and if we normalize the spherical vector to take value $1$ on identity, then $W \mapsto \,\!^\sigma W$ fixes this vector. This makes the local and global actions of $\sigma$ compatible.\\

Applying the functor $X \mapsto \Hom_{K_\infty}(\bigwedge^2 \mathfrak{p}^+ \otimes \mathfrak{p}^-, X)$ to the isomorphism $\Pi \overset{\sim}{\rightarrow} W(\Pi, \psi)$ defined by $\Psi \mapsto W_\Psi$ and using the fact that $\Hom_{K_\infty}(\bigwedge^2 \mathfrak{p}^+ \otimes \mathfrak{p}^-, \Pi_\infty^W) = H^3(\mathfrak{gsp}_{4}, K_\infty, \Pi_\infty^{2,1})$ (see Prop. \ref{dimension} and the proof of Lem. \ref{diff-form}), we obtain the canonical isomorphism
$$
H^3(\mathfrak{gsp}_{4}, K_\infty, \Pi_\infty^{2,1}) \otimes \Pi_f \overset{\sim}{\rightarrow} \Hom_{K_\infty}\left(\bigwedge^2 \mathfrak{p}^+ \otimes \mathfrak{p}^-, W(\Pi_\infty, \psi_\infty) \right) \otimes W(\Pi_f, \psi_f).
$$
Recall that $\Pi_\infty^W|_{\mrm{GSp}(4, \R)_+}=\Pi_\infty^{2,1} \oplus \Pi_\infty^{1,2}$. Let $W_{\Psi_\infty} \in W(\Pi_\infty, \psi_\infty)$ be the image of a highest weight vector $\Psi_\infty$ of the minimal $K_\infty$-type of $\Pi_\infty^{2,1}$ under the isomorphism $\Pi_\infty^W \simeq W(\Pi_\infty, \psi_\infty)$ induced by the isomorphism $\Pi  \overset{\sim}{\rightarrow} W(\Pi, \psi)$ fixed above. We normalize $W_{\Psi_\infty}$ so that 
\begin{equation} \label{normalization}
W_{X_{(-1,1)}\overline{\Psi}_\infty}(I_4)=1
\end{equation}
as in Prop. \ref{archi-comp}. Like in Lem. \ref{diff-form}, the normalized element $W_{\Psi_\infty}$ defines a normalized generator of the one-dimensional $\C$-vector space $\Hom_{K_\infty}\left(\bigwedge^2 \mathfrak{p}^+ \otimes \mathfrak{p}^-, W(\Pi_\infty, \psi_\infty) \right)$. Hence we obtain an isomorphism
$$
i: H^3(\mathfrak{gsp}_{4}, K_\infty, \Pi_\infty^{2,1}) \otimes \Pi_f \overset{\sim}{\rightarrow} W(\Pi_f, \psi_f).
$$
The left hand term has a $E$-structure $H^{2,1}_{dR}(\Pi_f)$ given by coherent cohomology of automorphic vector bundles (see \cite{harris1} 3. and in particular Prop. 3.3.9 for details). The following result can be proved exactly as \cite{grobner-sebastian} Prop. 3.3.12 be replacing the Bessel model in loc. cit. by the Whittaker model.

\begin{pro} \label{dR-W-periods} There exists $p(\Pi) \in \C^\times$, whose image in $\C^\times/E^\times$ is uniquely defined, such that for any $\sigma \in \mrm{Aut}(\C/E)$ the diagram 
$$
\begin{CD}
W(\Pi_f, \psi_f) @>p(\Pi)^{-1} i^{-1}>> H^{2,1}_{dR}(\Pi_f) \otimes_E \C\\
@V\tilde{\sigma}VV                                                                         @V1 \otimes \sigma VV\\
W(\Pi_f, \psi_f) @>p(\Pi)^{-1} i^{-1}>> H^{2,1}_{dR}(\Pi_f) \otimes_E \C
\end{CD}
$$
commutes.
\end{pro}

\begin{rems} Using the matrix 
$$
T'_{\sigma, p}=\begin{pmatrix}
t_{\sigma, p}^{-1} &  \\
 &  1\\
\end{pmatrix}
$$
instead of $T_{\sigma, p}$, a similar construction obviously exists for $\mrm{GL}(2)$ and we denote by $p(\sigma)$ the corresponding de Rham-Whittaker period attached to the Whittaker function $W_{\Phi_\infty}$ normalized as in Prop. \ref{archi-comp}. By the $q$-expansion principle, it could be possible to compute $p(\sigma)$ up to $E^\times$ multiples. In the case of the group $\mrm{GL}(2, F)$, where $F$ is a totally real quadratic number field, this kind of computation is performed in \cite{kings98} 3.5 . 
\end{rems}

\subsubsection{Rationality of local Rankin-Selberg integrals}

\begin{pro} \label{ram-int} Let $v$ be a non-archimedean place. Assume that $W_{\Psi_v}$ and $W_{\Phi_v}$ satisfy $\,\!^\sigma W_{\Psi_v}=W_{\Psi_v}$, $\,\!^\sigma W_{\Phi_v}=W_{\Phi_v}$ and $\sigma \circ \varphi_v=\varphi_v$ for any $\sigma \in \mrm{Aut}(\C/E)$. Then 
$$
\left. \frac{\mathcal{Z}_v(s, W_{\Psi_v}, W_{\Phi_v}, f_{\varphi_v})}{L(s, \Pi_v \times \sigma_v)} \right|_{s=1} \in E.
$$
Furthermore, there exist $W_{\Psi_v}$, $W_{\Phi_v}$ and $\varphi_v$ satisfying the condition above and such that
$$
\left. \frac{\mathcal{Z}_v(s, W_{\Psi_v}, W_{\Phi_v}, f_{\varphi_v})}{L(s, \Pi_v \times \sigma_v)} \right|_{s=1} \in E^\times
$$
and 
$$
\int_{\Q_v^2} \varphi_v(x, y) \,dx dy \in E^\times.
$$
\end{pro}

\begin{proof} Following \cite{soudry} p. 380, we write for $\re(s)$ big enough
$$
\mathcal{Z}_v(s, W_{\Psi_v}, W_{\Phi_v}, f_{\varphi_v})=\sum_{j=-\infty}^{+\infty} A_j(W_{\Psi_v}, W_{\Phi_v}, \varphi_v)p^{js}
$$
where 
$$
 A_j(W_{\Psi_v}, W_{\Phi_v}, \varphi_v)=\int_{(h_1, h_2) \in N_H(\Q_p) \backslash H(\Q_p), |\det h_1|=p^j}W_{\Psi_v}(h_1,h_2) W_{\Phi_v}(h_2) \varphi_v((0,1)h_1) dh.
$$
This integral is absolutely convergent and reduces to a finite sum because the integrated function has compact support modulo $N_H(\Q_p)$ in the set $\{ (h_1, h_2) \in H(\Q_p), |\det h_1|=p^j \}$ (\cite{casselman-shalika} Prop. 6.1) and is invariant by right translation by a sufficiently small open subgroup. Furthermore it vanishes for $j$ big enough. Let $\sigma \in \mrm{Aut}(\C)$. Fixing $j$ and computing integrals over $\{(h_1, h_2) \in N_H(\Q_p) \backslash H(\Q_p), |\det h_1|=p^j\}$, we have
$$
 A_j(\,\!^\sigma W_{\Psi_v}, \,\!^\sigma W_{\Phi_v}, \sigma \circ \varphi_v)
$$
\begin{eqnarray*}
&=& \int \sigma W_{\Psi_v} \left(\begin{pmatrix}
t_{\sigma}^{-3} & &  & \\
 &  t_{\sigma}^{-2} &  & \\
 &  & 1 & \\
 & &  & t_{\sigma}^{-1}\\
\end{pmatrix} (h_1,h_2) \right) \sigma W_{\Phi_v}\left(\begin{pmatrix}
t_{\sigma}^{-1} &  \\
 &  1\\
\end{pmatrix}  h_2 \right) \sigma \varphi_v((0,1)h_1) dh\\
&=& \int \sigma W_{\Psi_v} \left(\begin{pmatrix}
t_{\sigma}^{-3} & &  & \\
 &  t_{\sigma}^{-2} &  & \\
 &  & 1 & \\
 & &  & t_{\sigma}^{-1}\\
\end{pmatrix} (h_1,h_2) \right) \sigma W_{\Phi_v}\left(\begin{pmatrix}
t_{\sigma}^{-2} &  \\
 &  t_{\sigma}^{-1}\\
\end{pmatrix}  h_2 \right) \sigma \varphi_v((0,1)h_1) dh\\
&=& \sigma A_j(W_{\Psi_v}, W_{\Phi_v}, \varphi_v)
\end{eqnarray*}
where the first equality follows from the definition of $\,\!^\sigma W_{\Psi_v}$ and  $\,\!^\sigma W_{\Phi_v}$, the second from the fact that $W_{\Phi_v}$ has trivial central character, and the last from an easy change of variable. In particular, under the assumptions of the Proposition, the function $\mathcal{Z}_v(s, W_{\Psi_v}, W_{\Phi_v}, f_{\varphi_v})$ is an element of $E((p^{-s}))$ and it follows from \cite{soudry} Lem. 3.2 that this function extends to an element of $E(p^{-s})$. This implies the first statement. The non-vanishing of
$$
\left. \frac{\mathcal{Z}_v(s, W_{\Psi_v}, W_{\Phi_v}, f_{\varphi_v})}{L(s, \Pi_v \times \sigma_v)} \right|_{s=1}
$$
for well chosen functions $W_{\Psi_v}$, $W_{\Phi_v}$ and $\varphi_v$ satisfying the assumptions of the Proposition is a direct consequence of Prop. \ref{ram-int-soudry} and of the fact that functions satisfying $\,\!^\sigma W_{\Psi_v}=W_{\Psi_v}$, $\,\!^\sigma W_{\Phi_v}=W_{\Phi_v}$ and $\sigma \circ \varphi_v$ for any $\sigma \in \mrm{Aut}(\C/E)$ define $E$-structures in the corresponding $\C$-vector spaces. The existence of $W_{\Psi_v}$, $W_{\Phi_v}$ and $\varphi_v$ such that in addition one has 
$$
\int_{\Q_v^2} \varphi_v(x, y) \,dx dy \in E^\times
$$
is a consequence of the following elementary result: let $k$ be a field and let $\mathcal{V}$ be a $k$-vector space with two non-zero linear functionals $l_1, l_2: \mathcal{V} \rightarrow k$; then $\ker l_1 \cup \ker l_2 \neq \mathcal{V}$.
\end{proof}

\section{Proof of the main result} \label{pfs}

Let $p(\Pi)$, resp. $p(\sigma)$, be the de Rham-Whittaker periods attached to $\Pi$, resp. $\sigma$, defined and normalized in section \ref{dR-W}. Let $p(\Pi \times \sigma)$ denote the product $p(\Pi)p(\sigma)$.

\begin{thm} Assume that $L(s, \Pi \times \sigma)$ has a pole at $s=1$ and that $\Pi$ and $\sigma$ have trivial central characters. Then
$$
\mathcal{Z}(\Pi_f \times \sigma_f)=p(\Pi \times \sigma)\Res_{s=1}L(s, \Pi \times \sigma) \mathcal{D}(\Pi_f \times \sigma_f).
$$
\end{thm}

\begin{proof} Let $\Psi=\bigotimes'_v \Psi_v$ be factorizable a cusp form in the representation space of $\Pi$ such that $\Psi_\infty$ is a highest weight vector of the minimal $K_\infty$-type normalized as in \eqref{normalization} and such that for any non-archimedean place $v$, the Whittaker function $W_{\Psi_v}$ is satisfies $\,\!^\sigma W_{\Psi_v}=W_{\Psi_v}$ for any element $\sigma$ of $\mrm{Aut}(\C/E)$. Let $\Phi=\bigotimes'_v \Phi_v$ be a factorizable cusp form in the representation space of $\sigma$ such that $\Phi_\infty$ is a generator of the minimal $L_\infty$-type of $\sigma_\infty$ normalized in such a way that $W_{\Phi_\infty}(I_2)=c_2$ as in Prop. \ref{archi-comp} and such that for any non-archimedean place $v$, the Whittaker function $W_{\Phi_v}$ satisfies $\,\!^\sigma W_{\Phi_v}=W_{\Phi_v}$ for any $\sigma \in \mrm{Aut}(\C/E)$. It follows from Prop. \ref{mult} and the Hodge decomposition (\ref{hodge}) that we have the isomorphism $F^2 M_{dR}(\Pi_f)_{\C} \simeq M^{2, 1}(\Pi_f)$ induced by the comparison isomorphism. Hence, by Prop. \ref{dR-W-periods}, the vector $p(\Pi) \omega_\Psi$ is a generator of $F^2 M_{dR}(\Pi_f)$ and, similarly, the vector $p(\sigma) \eta_\Phi$ is a generator of $F^1 M_{dR}(\sigma_f)$. According to Cor. \ref{un-corollaire}, we have
$$
\mathcal{Z}(\Pi_f \times \sigma_f) = \pi^8 p(\Pi \times \sigma) \langle v_\mathcal{Z}, \overline{\omega}_\Psi \otimes \eta_\Phi \rangle_{B, \C} \mathcal{D}(\Pi_f \times \sigma_f).
$$
Applying to Prop. \ref{coho-period} and Cor. \ref{residue} we have
$$
\mathcal{Z}(\Pi_f \times \sigma_f) = \pi^8 p(\Pi \times \sigma) \widehat{\varphi}(0)^{-1} \Res_{s=1} \mathcal{Z}\left(s, X_{(-1,1)}\Psi, \Phi, f_\varphi \right) \mathcal{D}(\Pi_f \times \sigma_f) 
$$
where in the last equality $\varphi=\prod_v \varphi_v$ is any factorizable Schwartz-Bruhat function on $\A^2$ whose archimedean component is $\varphi_\infty(x,y)=\exp(-\pi(x^2+y^2))$ and whose non-archimedean components at ramified places are given by Prop. \ref{ram-int}. For such a choice of $\varphi$, we have $\widehat{\varphi}(0) \in E^\times$. Hence the statement follows from the combination of Prop. \ref{archi-comp} and Prop. \ref{ram-int}.
\end{proof}


\begin{thebibliography}{9999}

\bibitem[AS06]{asgari-shahidi} M. Asgari, F. Shahidi, {\it Generic transfer from $\mrm{GSp(4)}$ to $\mrm{GL}(4)$}, Compositio Math. 142, (2006), 541-550.


\bibitem[Be86]{beilinson2} A. A. Beilinson, {\it Notes on absolute Hodge cohomology}, Applications of algebraic $K$-theory to algebraic geometry and number theory, Part I, II, 35-68, Contemp. Math. 55, Amer. Math. Soc., Providence, RI, (1986).


\bibitem[BHR94]{bhr} D. Blasius, M. Harris, D. Ramakrishnan, {\it Coherent cohomology, limits of discrete series and Galois conjugation}, Duke Math. Journ. 73 (1994), no. 3, 647-685.

\bibitem[Bo81]{borel0} A. Borel, {\it Stable real cohomology of arithmetic groups II}, Manifolds and Lie groups, Papers in honor of Y. Matsushima, Prog. in Math., Birkh\"{a}user, Boston, (1981), 21-55.


\bibitem[BoW80]{borel-wallach} A. Borel, N. R. Wallach, {\it Continous cohomology, discrete subgroups and representations of reductive groups}, Ann. of Math. Studies 94, Princeton University Press, Princeton N.J., (1980).



\bibitem[Bu97]{bump} D. Bump, {\it Automorphic forms and representations}, Cambridge studies in advanced mathematics 55, Cambridge university press, Cambridge, (1997), xiv+574.


\bibitem[CS80]{casselman-shalika} W. Casselman, J. Shalika, {\it The unramified principal series of $p$-adic groups II. The Whittaker function}, Compositio Math. 41, no. 2, (1980), 207-231.


\bibitem[De79]{valeurs-deligne} P. Deligne, {\it Valeurs de fonctions $L$ et p\'eriodes d'int\'egrales}, Proc. Sympos. Pure Math. 33, part 2, (1979), 313-346.

\bibitem[DS]{den-scholl} C. Deninger, A. Scholl, {\it The Beilinson conjectures}, in $L$-functions and arithmetic (Durham, 1989), 173-209, London Math. Soc. Lecture Note Ser., 153, Cambridge Univ. Press, Cambridge, (1991).

\bibitem[GH78]{griffiths-harris} P. Griffiths, J. Harris, {\it Principles of algebraic geometry}, Pure and applied Math., Wiley-Interscience, New-York, (1978), xii+813.


\bibitem[GrS17]{grobner-sebastian} H. Grobner, R. Sebastain, {\it Period relations for cusp forms of $\mrm{GSp}_4$ and general aspects for modular symbols}, Forum Math., (2017).


\bibitem[H94]{harris1} M. Harris, {\it Hodge-de Rham structures and periods of automorphic forms}, Motives, Part 2, Proc. Sympos. Pure Math. 55, Amer. Math. Soc., Providence RI, (1994), 573-624.

\bibitem[H04]{occult} M. Harris, {\it Occult period invariants and critical values of the degree four $L$-function of $\mathrm{GSp}(4)$}, Contributions to automorphic forms, geometry and number theory 331-354, John Hopkins Univ. Press, Baltimore, MD, (2004).

\bibitem[HK92]{harris-kudla} M. Harris, S. Kudla, {\it Arithmetic automorphic forms for the nonholomorphic discrete series of $\mrm{GSp}(2)$}, Duke Math. J. 66, no. 1, (1992), 59-121.

\bibitem[J72]{jacquet} H. Jacquet, {\it Automorphic forms on $\mrm{GL}(2)$ II}, Lecture Notes in Mathematics 278, Springer-Verlag, Berlin-New York, (1972), xiii+142.

\bibitem[JL70]{jacquet-langlands} H. Jacquet, R. Langlands, {\it Automorphic forms on $\mrm{GL}(2)$ I}, Lecture Notes in Mathematics 114, Springer-Verlag, Berlin-New York, (1970), vii+548.


\bibitem[JSo07]{jiang-soudry} D. Jiang, D. Soudry, {\it The multiplicity one theorem for generic automorphic forms of $\mrm{GSp}(4)$}, Pacific Journ. of Math. Vol. 229, no. 2, (2007), 381-388.




\bibitem[K98]{kings98} G. Kings, {\it Higher regulators, Hilbert modular surfaces and special values of $L$-functions}, Duke Math. J. 92, no. 1, (1998), 61-127.

\bibitem[La05]{laumon} G. Laumon, {\it Fonctions z\^etas des vari\'et\'es de Siegel de dimension $3$}, Ast\'erisque No. 302, (2005), 1-66.






\bibitem[Le17]{lemma2} F. Lemma, {\it On higher regulators of Siegel threefolds II: the connection to the special value}, Compositio Math. 153, no. 5, (2017), 889-946.


\bibitem[Mo09]{moriyama1} T. Moriyama, {\it $L$-functions for $\mathrm{GSp}(2) \times \mathrm{GL}(2)$: archimedean theory and applications}, Canad. J. Math. 61 (2), (2009), 395-426.

\bibitem[Ne94]{nekovar} J. Nekovar, {\it Beilinson's conjectures}, Motives, Proc. Sympos. Pure Math. 55, Part 1, Amer. Math Soc., Providence, (1994), 537-570.

\bibitem[PSS84]{pssoudry} I. I. Piatetski-Shapiro, D. Soudry, {\it $L$ and $\varepsilon$ functions for $\mrm{GSp}(4) \times \mrm{GL}(2)$}, Proc. Nati. Acad. Sci. USA, Vol. 81, No. 12, (1984), 3924-3927.

\bibitem[PSS87]{pssoudry2} I. I. Piatetski-Shapiro, D. Soudry, {\it On a correspondence of automorphic forms on orthogonal groups of order five}, J. Math. Pures Appl. (9) 66, (1987), no. 4, 407-436.

\bibitem[P90]{pink} R. Pink, {\it Arithmetical compactifications of mixed Shimura varieties}, Bonner Mathematische Schriften 209, (1990).

\bibitem[Sc88]{schneider} P. Schneider, {\it Introduction to the Beilinson conjectures}, Beilinson's conjectures on special values of $L$-functions, Perspect. Math. 4, Academic Press, Boston, MA, (1988), 1-35.

\bibitem[Sch11]{scholze} P. Scholze, {\it The Langlands-Kottwitz approach for the modular curve}, IMRN no. 15, (2011), 3368-3425.

\bibitem[Sha81]{shahidi} F. Shahidi, {\it On certain $L$-functions}, Amer. Journ. of Math., vol. 103, no. 2, (1981), 297-355.

\bibitem[Sh74]{shalika} J. A. Shalika, {\it The multiplicity one theorem for $\mathrm{GL}_n$}, Annals of Math., Vol. 100, no.2, (1974), 171-193.


\bibitem[So84]{soudry} D. Soudry, {\it The $L$ and $\gamma$ factors for generic representations of $\mrm{GSp}(4, k) \times \mrm{GL}(2, k)$ over a local non-archimedean field $k$}, Duke Math. J. 51, no. 2, (1984), 355-394. 

\bibitem[Ta93]{taylor} R. Taylor, {\it On the $l$-adic cohomology of Siegel threefolds}, Invent. Math. 114, no. 1, (1993), 289-310.



\bibitem[Wa85]{waldspurger} J.-L. Waldspurger, {\it Quelques propri\'et\'es arithm\'etiques de certaines formes automorphes sur $\mrm{GL}(2)$}, Compositio Math. 54, no. 2, (1985), 121-171.


\bibitem[We05]{weissauer1} R. Weissauer, {\it Four dimensional Galois representations}, Formes automorphes II. Le cas du groupe GSp(4), Ast\'erisque No. 302, (2005), 67-150.

\bibitem[We09]{weissauer2} R. Weissauer, {\it Endoscopy for $\mrm{GSp}(4)$ and the cohomology of Siegel modular threefolds}, Lecture Notes in Mathematics 1968, Springer-Verlag, Berlin, (2009), xviii+368.

\end{thebibliography}
\end{document}